\documentclass[14pt]{amsart}
\usepackage[english]{babel}
\usepackage{graphicx}
\usepackage[T1]{fontenc}
\usepackage{color}
\usepackage{bigints}
\usepackage{soul}
\usepackage{mathtools}
\usepackage{tikz}
\usepackage{amsmath}
\usetikzlibrary{arrows.meta}
\providecommand{\gtrless}{
	\mathrel{
		\smash{
			\vcenter{
				\offinterlineskip 
				\ialign{
					\hfil##\hfil\cr 
					$>$\cr 
					\noalign{\kern-.3ex}
					$<$\cr 
				}
			}
		}
		\vphantom{>}
	}
}
\usetikzlibrary{shapes.geometric, arrows.meta, positioning}

\tikzstyle{startstop} = [rectangle, rounded corners, minimum width=3cm, minimum height=1cm, text centered, draw=black, fill=red!30]
\tikzstyle{process} = [rectangle, minimum width=3cm, minimum height=1cm, text centered, draw=black, fill=blue!30]
\tikzstyle{decision} = [diamond, minimum width=3.5cm, minimum height=1cm, text centered, draw=black, fill=green!30]
\tikzstyle{arrow} = [thick,->,>=stealth]

\newtheorem{Thm}{Theorem}
\newtheorem{Prop}[Thm]{Proposition}

\newtheorem{Lem}{Lemma}

\newtheorem*{Prop*}{Proposition}
\newtheorem*{Cor*}{Corollary}
\newtheorem*{Thm*}{Theorem}
\usepackage[backref=page]{hyperref}
\usepackage[utf8]{inputenc}
\usepackage[english]{babel}
\usepackage{etoolbox}
\DeclarePairedDelimiter\floor{\lfloor}{\rfloor}
\apptocmd{\sloppy}{\hbadness 10000\relax}{}{}
\hypersetup{
	colorlinks=true,
	linkcolor=blue,
	urlcolor=cyan,
}

\begin{document}
	\title[Fractional Borg-Levinson Problem]{On Fractional Borg-Levinson Problem}
	\author[Saumyajit Das]{Saumyajit Das}
	\address{S.D.: Department of Mathematics, Harish-Chandra Research Institute, A CI of Homi Bhabha National Institute, Chhatnag Road, Jhunsi, Allahabad 211 019, India.}
	\email{saumyajit.math.das@gmail.com}
	
	\author[Tuhin Ghosh]{Tuhin Ghosh}
	\address{T.G.: Department of Mathematics, Harish-Chandra Research Institute, A CI of Homi Bhabha National Institute, Chhatnag Road, Jhunsi, Allahabad 211 019, India.}
	\email{tuhinghosh@hri.res.in}
	\begin{abstract}
		In this article, we investigate the fractional Borg-Levinson problem, an inverse spectral problem focused on recovering potentials from boundary spectral data. We demonstrate that, within an admissible class of potentials, the potential is uniquely determined by the given data. However, for technical reasons, we restrict the fractional exponent to the interval $(\frac12,1)$. The admissible class is explicitly defined in our calculations and depends solely on the domain and the spatial dimension.
	\end{abstract}

	
	
	\maketitle
	\bibliographystyle{alpha}
	
	\section{Introduction}
	Borg-Levinson problem has a distinguished place in inverse spectral theory, focusing on the determination of unknown potentials within differential operators from spectral data. It began with Borg(1946) \cite{borg1946umkehrung} and Levinsion (1949) \cite{levinson1949inverse}. It was a one dimensional problem described below:
	\begin{Thm}[Borg-Levinson]\label{local borg-levinson}
		For $\lambda\in\mathbb{R}$ and $q_i\in \mathrm{L}^{\infty}((0,1);\mathbb{R})$, $i=1,2$, let $u_i(\cdot,\lambda)$ be the $H^2((0,1))$ solution to the initial value problem: 
		\begin{equation*}
			\left \{
			\begin{aligned}
				\left(-\frac{d^2}{dx^2}+q_i\right) u_i(x,\lambda)=& \lambda u_i(x,\lambda), \qquad x\in(0,1)\\
				u_i(0,\lambda)=& 0\\
				\frac{d}{dx}u_i(0,\lambda)=&1. 
			\end{aligned}
			\right .
		\end{equation*}
		Let $\{\lambda_{i,n}, \, n\in\mathbb{N}\}$ be the non-decreasing sequence of the Dirichlet eigenvalues associated with the operator $\displaystyle{-\frac{d^2}{dx^2}+q_i}$, obtained by imposing 
		\[
		u_1(1,\lambda_{i,n})=0.
		\]
		Then if we assume the equality on the boundary spectral data i.e.,
		\[
		\lambda_{1,n}=\lambda_{2,n} \ \mbox{and}\ \Vert u_1(\cdot,\lambda_{1,n})\Vert_{\mathrm{L}^2((0,1))}= \Vert u_2(\cdot,\lambda_{2,n})\Vert_{\mathrm{L}^2((0,1))}, \quad \forall \, n\in\mathbb{N},
		\]
		we have
		\[
		q_1\equiv q_2.
		\] 
	\end{Thm} 
	The extra condition is significant in the context of identifying the potentials. The spectrum alone can not determine the potentials as the spectrum does not discriminate between symmetric potentials. This can be seen by noticing that we have
	\[
	U \left( -\frac{d^2}{d ^2}+q_1\right)U^{-1}= \mathcal{A}_{U_1},
	\]
	where we set $\displaystyle{U(f(x))=f(1-x)}$ for all $f\in\mathrm{L}^2((0,1))$. Since $U$ is a unitary operator in $\mathrm{L}^2((0,1))$, the two operators $\displaystyle{\left( -\frac{d^2}{d ^2}+q_1\right)}$ and $\displaystyle{\mathcal{A}_{U_1}}$ have the same spectral \cite{soccorsi:hal-03571903}. One cannot distinguish between the potentials $q_1$ and $U(q_1)U^{-1}$ based solely on the two spectra, unless $q_1$  is a symmetric potential \cite{soccorsi:hal-03571903}. Later on Gel'fand and Levitan \cite{gel1951determination} proved that the uniqueness of potential is still valid upon substituting $u'_i(1,\lambda_{i,n})$ for $\displaystyle{\Vert u_1(\cdot,\lambda_{1,n})\Vert_{\mathrm{L}^2((0,1))}}$, for $i=1,2$, in Borg-Levinson problem \ref{local borg-levinson}. We like to emphasize the fact that the techniques used in \cite{borg1946umkehrung}, \cite{levinson1949inverse}, \cite{gel1951determination} were highly based on one dimensional techniques. The lead towards multidimension starts with techniques developed in \cite{belishev2007recent}, \cite{ramm2005inverse}. These techniques are mainly based on density of solution space and boundary control methods. In \cite{nachman1988n} (and independently in \cite{novikov1988multidimensional}) authors proved the multidimensional version of the Borg-Levinson problem. We describe the result below:
	\begin{Thm}[Nachman-Sylvester-Uhlmann/ Novikov]\label{Uhlmann}
		Let $\Omega$ be a smooth, bounded, connected domain in $\mathbb{R}^N$. For $\lambda\in\mathbb{R}$ and $q_i\in \mathrm{L}^{\infty}(\Omega)$, $i=1,2$, let $u_i(\cdot,\lambda)$ be the $H^2(\Omega)$ solution to the following boundary value problem: 
		\begin{equation*}
			\left \{
			\begin{aligned}
				\left(-\Delta+q_i\right) u_i(x,\lambda)=& \lambda u_i(x,\lambda), \qquad \mbox{in} \ \Omega\\
				u_i(\cdot,\lambda)=& 0 \qquad \qquad  \quad \mbox{on} \ \partial\Omega.
			\end{aligned}
			\right .
		\end{equation*}
		Let $\{\lambda_{i,n}, \, n\in\mathbb{N}\}$ the non-decreasing sequence of the Dirichlet eigenvalues associated with the operator $\displaystyle{-\Delta+q_i}$.
		Then, if we assume the equality on the boundary spectral data i.e.,
		\[
		\lambda_{1,n}=\lambda_{2,n} \ \mbox{and}\  \partial_{\nu} u_{1}(x,\lambda_{1,n})|_{\partial\Omega}=\partial_{\nu} u_{2}(x,\lambda_{2,n})|_{\partial\Omega}, \quad \forall \, n\in\mathbb{N},
		\]
		we have
		\[
		q_1\equiv q_2.
		\] 
	\end{Thm} 
	Here, $\partial_{\nu}$ denotes the outer normal derivative on the boundary $\partial\Omega$, where $\nu$ is the outer normal at a point on $\partial\Omega$. This result is generalized in \cite{isozaki1989some} and \cite{choulli2009introduction}. The generalized result ensures that the multidimensional Borg–Levinson theorem remains valid even when the equality of the boundary spectral data holds only in the tail of the boundary spectra, i.e.,
	\[
	\lambda_{1,n}=\lambda_{2,n} \ \mbox{and}\  \partial_{\nu} u_{1}(x,\lambda_{1,n})|_{\partial\Omega}=\partial_{\nu} u_{2}(x,\lambda_{1,n})|_{\partial\Omega}, \quad \mbox{for} \ n\geq k_0\in\mathbb{N}.
	\]  
	In \cite{choulli2013stability}, \cite{kavian2015uniqueness}, authors determine the potentials by analyzing the asymptotic behaviour of the boundary spectral data. They proved the multidimensional Borg-Levinson problem under the following assumptions:
	\[
	\lim_{n\to+\infty} \big(\lambda_{1,n}-\lambda_{2,n}\big)=0 \ \ \mbox{and}\ \sum\limits_{n=1}^{+\infty} \Vert \partial_{\nu} u_{1}(x,\lambda_{1,n})-\partial_{\nu} u_{2}(x,\lambda_{2,n})\Vert^2_{\mathrm{L}^2(\partial\Omega)} <+\infty. 
	\] 
	The multidimensional Borg-Levinsion problem has been studied in many different kind of settings. In \cite{carlson1994inverse}, \cite{paivarinta2002n}, \cite{pohjola2018multidimensional}, authors considered discontinuous and unbounded potentials. The Borg-Levinson problem with partial Neumann spectral data can be found in \cite{bellassoued2009stability}, where a log-stability estimate for electric potentials which are known in a neighborhood of a boundary, is established with respect to the boundary spectral data measured on an arbitrary non-empty open subset of the boundary. For a comprehensive survey, we like to refer the readers to \cite{soccorsi:hal-03571903}.  
	
	In this article, we will extend the multidimensional Borg-Levinson problem in fractional settings. However, for technical reasons, we restrict the fractional exponent to the interval $\left(\frac12,1\right)$. Additionally, we also assume the potential is small, non-negative and exhibits small growth. Before describing our result, we will like to describe the boundary spectral data for fractional Borg-Levinson problem. We will assume $\Omega$ is a smooth connected domain. Let $q_1,q_2\in C_c^{\infty}(\Omega)$. Consider the operators \((-\Delta)^a + q_1\) and \((-\Delta)^a + q_2\), and their corresponding Dirichlet spectral datum
	\[
	\left\{ \lambda_{n,q_1}, \frac{\phi_{n,q_1}}{d^a} \; : \; n \in \mathbb{N} \right\}, \quad
	\left\{ \lambda_{n,q_2}, \frac{\phi_{n,q_2}}{d^a} \; : \; n \in \mathbb{N} \right\},
	\]
	respectively, where for \(i = 1, 2\), the eigenvalue \(\lambda_{n,q_i}\) and normalized eigenfunction \(\phi_{n,q_i}\) satisfy the following equation: \  $\forall \, n\in\mathbb{N}$
	\begin{equation}\label{1st elliptic eigenvalue problem}
		\left \{
		\begin{aligned}
			(-\Delta)^a \phi_{n,q_i}+q_i \phi_{n,q_i}=& \lambda_{n,q_i}\phi_{n,q_i} \qquad \mbox{in}\ \Omega\\
			\phi_{n,q_i}=& 0 \qquad \qquad \mbox{in}\ \mathbb{R}^N\setminus\Omega,
		\end{aligned}
		\right .
	\end{equation}
	and $d:\overline{\Omega}\to\mathbb{R}$ is the following distance function
	\[
	d= \inf_{y\in\overline{\Omega}}\Vert x-y\Vert_2, 
	\]
	where $\Vert \cdot\Vert_2$ is the classical Euclidean norm in $\mathbb{R}^N$. The Dirichlet spectral data $\displaystyle{\frac{\phi_{n,q_i}}{d^a}}$ is well defined and belongs to the space $H^{a-\frac12}(\partial\Omega)$ \cite{grubb2017fractional}, \cite{grubb2018green}, \cite{ros2014dirichlet}. Furthermore, in \cite{fernandez2024integro}, \cite{abels2023fractional}, \cite{abatangelo2017very} authors derived the following result:
	\[
	\left\Vert \frac{\phi_{n,q_i}}{d^a}\right\Vert_{\mathrm{H}^{a-\frac12}(\partial\Omega)} \leq C(\Omega)|\lambda_n|\Vert \phi_{n,q_i}\Vert_{\mathrm{L}^2(\Omega)}\leq C(\Omega)\vert\lambda_n\vert, \ \forall \, n\in\mathbb{N} \ \mbox{and for}\ i=1,2,
	\]
	where $C(\Omega)$ is a positive constant depending on the domain only. For $a\in(\frac12, 1)$, the continuous inclusion $H^{a-\frac12}(\partial\Omega)\hookrightarrow \mathrm{L}^2(\partial\Omega)$ yields
	\begin{align}\label{eigen value regularity}
		\left\Vert \frac{\phi_{n,q_i}}{d^a}\right\Vert_{\mathrm{L}^2(\partial\Omega)} \leq C(\Omega)|\lambda_n|\Vert \phi_{n,q_i}\Vert_{\mathrm{L}^2(\Omega)}\leq C(\Omega)|\lambda_n|, \ \forall \, n\in\mathbb{N} \ \mbox{and for}\ i=1,2,
	\end{align}	
	where $C(\Omega)$ is a positive constant depending on domain only.
	
	\textbf{The \emph{fractional Borg–Levinson problem} is stated as follows:}
	
	Suppose the eigenvalues and the Dirichlet datum of the fractional Schrödinger operators
	\[
	(-\Delta)^a + q_1 \quad \text{and} \quad (-\Delta)^a + q_2
	\]
	are identical, i.e.,
	\begin{align}\label{Borg Levinson fractional criteria}
		\lambda_{n,q_1} = \lambda_{n,q_2} =: \lambda_n, \quad \text{and} \quad \frac{\phi_{n,q_1}}{d^a} = \frac{\phi_{n,q_2}}{d^a} \quad \text{on } \partial\Omega, \quad \text{for all } n \in \mathbb{N}.
	\end{align}
	The question is: \emph{Does this imply that the potentials are equal, i.e., \( q_1 \equiv q_2 \) in \( \Omega \)}?
	\newline
	For both the operators, the discrete Dirichlet eigenspectra are ordered monotonically, as described below \cite{pazy2012semigroups}. 
	\begin{align*}
		\lambda_{1,q_1}< \lambda_{2,q_1} \leq \cdots <\lambda_{n,q_1}<\cdots +\infty\\
		\lambda_{1,q_2}< \lambda_{2,q_2} \leq \cdots <\lambda_{n,q_2}<\cdots +\infty.
	\end{align*}
	Due to the equality of eigenvalues we have that 
	\[
	\lambda_1<\lambda_2\leq \cdots\leq \lambda_n\leq \cdots+\infty.
	\]   
	We now proceed to describe our result in detail below. However we assume a certain smallness assumption on one of the potential. 
	\begin{Thm}\label{fractional borg-levinson theorem}
		Let $\Omega$ be a smooth, bounded, connected domain in $\mathbb{R}^N$ and let $a\in\left(\frac12,1\right)$. For $\lambda\in\mathbb{R}$ and $q_i\in \mathrm{L}^{\infty}(\Omega)$, $i=1,2$. Let $\displaystyle{\theta< \frac{1}{2}\left(1+C_{HS}\left(\frac N2+R\right)\right)^{-1}}$, where $C_{HS}$ is the Hardy-Sobolev constant on the domain $\Omega$ and $R:=\max\{\Vert x\Vert, x\in\Omega\}$. Let the potential $q_1$ is non-negative and $|q_1(x)|,|\nabla q_1(x)|\leq \theta$ for all $x\in\Omega$. Let $u_i(\cdot,\lambda)$ be the $H^a(\Omega)$ solution to the following problem: 
		\begin{equation*}
			\left \{
			\begin{aligned}
				\left((-\Delta)^a+q_i\right) u_i(x,\lambda)=& \lambda u_i(x,\lambda), \qquad \mbox{in} \ \Omega\\
				u_i(\cdot,\lambda)=& 0 \qquad \qquad  \quad \mbox{in} \ \mathbb{R}^N\setminus\Omega.
			\end{aligned}
			\right .
		\end{equation*}
		Let $\{\lambda_{n,q_i}, \, n\in\mathbb{N}\}$ be the non-decreasing sequence of the Dirichlet eigenvalues associated with the operator $\displaystyle{(-\Delta)^a+q_i}$ and $\{\phi_{n,q_i}, \, n\in\mathbb{N}\}$ be the corresponding normalized eigen space.
		Then, if we assume the equality on the boundary spectral datum i.e.,
		\[
		\lambda_{n,q_1}=\lambda_{n,q_2}=\lambda_n \ \mbox{and}\  \frac{\phi_{n,q_1}}{d^a}\Big|_{\partial\Omega}=\frac{\phi_{n,q_1}}{d^a}\Big|_{\partial\Omega},
		\]
		we have
		\[
		q_1\equiv q_2.
		\] 
	\end{Thm} 
	\subsection{Plan of the paper:}
	To established our result, we actually borrowed some ideas from \cite{nachman1988n}, \cite{sylvester1987global} and \cite{katchalov2004equivalence}. We study the Gel'fand problem [for details see \cite{nachman1988n},  \cite{katchalov2004equivalence}] corresponding to the fractional Borg-Levinson problem \ref{fractional borg-levinson theorem} and demonstrate that the \emph{Dirichlet to Neumann (DN) maps} (defined ellaborately in subsection \ref{Definition of DN map} ) corresponding to the two operators are equal. Our contribution here is to extend the local results in  \cite{nachman1988n}, \cite{sylvester1987global}, \cite{katchalov2004equivalence} to fractional settings. The equality of elliptic \emph{Dirichlet to Neumann (DN) maps} has been established in section \ref{equality of dn maps}, specifically in Proposition \ref{equality of elliptic DN map}. Next, as in for the local case [ see \cite{nachman1988n}, \cite{katchalov2004equivalence}, \cite{isakov2006inverse}, \cite{isakov1991inverse}], we derive that the parabolic  \emph{Dirichlet to Neumann (DN) maps} are also equal for the fractional case. It has been proved in section \ref{Borg-Levinson problem to parabolic inverse problem}, specifically in Proposition \ref{equality of parabolic DN map.}. We turn the parabolic problem into fractional elliptic one with the help of Laplace transformation as described in section \ref{Borg-Levinson problem to parabolic inverse problem}. Hence, we reduced the problem in the level of solving the non-local parabolic inverse problem where the potentials generate same \emph{Dirichlet to Neumann (DN) maps}. The elliptic counterpart of this non-local parabolic inverse problem is well studied and comes under inverse problem on Schr\'odinger operator \cite{ghosh2020calderon}, \cite{ghosh2017calderon}, \cite{ghosh2020uniqueness}, \cite{salo2017fractional}, \cite{ghosh2021non}. The fractional elliptic counterpart depends on the density of the solution space and unique continuation principle as described in \cite{ghosh2020uniqueness}. Here, our proof of the inverse problem for the non-local parabolic equation, based on the equality of the \emph{Dirichlet-to-Neumann (DN) maps}, relies on the density of the solution space as described in \cite{das2025boundarycontrolcalderontype}. However, for the sake of the readers we give the proof of the density result in Appendix \ref{Appendix}. The density result is motivated from \cite{dipierro2019local}, \cite{Ruland2017QuantitativeAP}, \cite{zuazua}, \cite{lions1992remarks}. The proof relies on boundary control technique and observability estimate as described in \cite{das2025boundarycontrolcalderontype}, \cite{biccari2025boundary}. We present the proof of the non-local parabolic inverse problem under the assumption on the equality of the \emph{Dirichlet-to-Neumann (DN) maps} in section \ref{Parabolic inverse problem Neumann data}, specifically in Proposition \ref{non local parabolic inverse problem}. Finally, these lead to the proof of Theorem \ref{fractional borg-levinson theorem} (described in section \ref{Parabolic inverse problem Neumann data}). We want to emphasize the fact that in our calculations $a\in\left(\frac12,1\right)$ is needed. The flow chart in the next page will help the readers to follow our proof sequentially.  
	\newpage
	
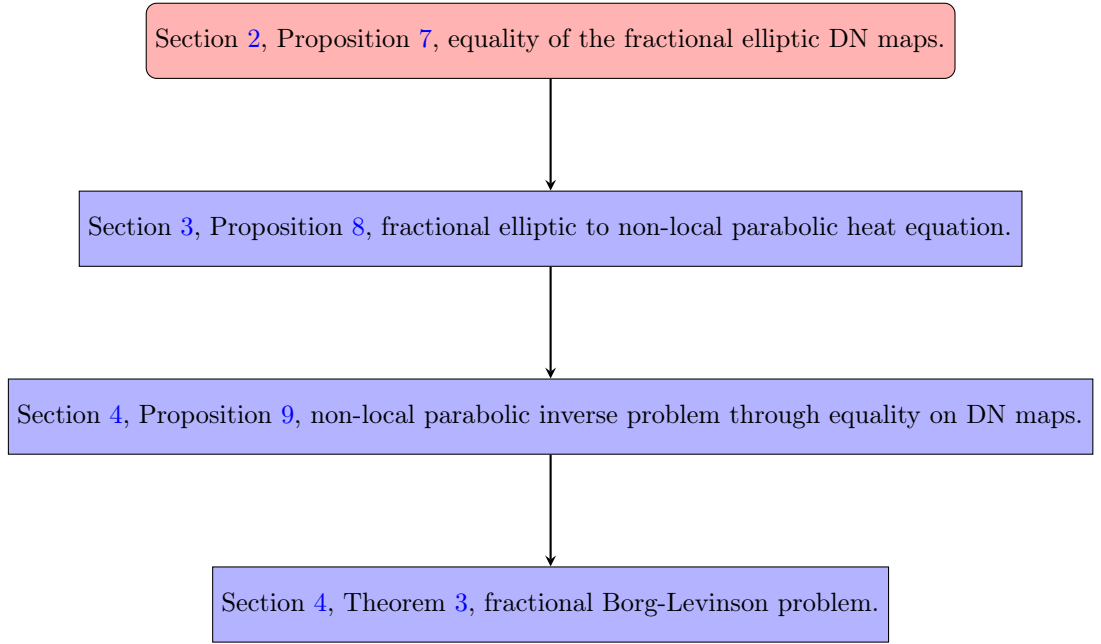
\begin{tikzpicture}[
		node distance=1.5cm,
		startstop/.style={
			text width=12cm,
			align=center
		},
		process/.style={
			text width=12cm,
			align=center
		},
		arrow/.style={
			->,
			thin,
			>=stealth
		}
		]
		
		\node (start) [startstop]
		{Section \ref{equality of dn maps}, Proposition
			\ref{equality of elliptic DN map}, equality of the fractional
			elliptic DN maps.};
		
		\node (input) [process, below of=start]
		{Section \ref{Borg-Levinson problem to parabolic inverse problem},
			Proposition \ref{equality of parabolic DN map.}, fractional elliptic
			to non-local parabolic heat equation.};
		
		\node (process) [process, below of=input]
		{Section \ref{Parabolic inverse problem Neumann data},
			Proposition \ref{non local parabolic inverse problem}, non-local
			parabolic inverse problem through equality on DN maps.};
		
		\node (output) [process, below of=process]
		{Section \ref{Parabolic inverse problem Neumann data},
			Theorem \ref{fractional borg-levinson theorem}, fractional
			Borg-Levinson problem.};
		
		\draw [arrow] (start) -- (input);
		\draw [arrow] (input) -- (process);
		\draw [arrow] (process) -- (output);
		
	\end{tikzpicture}
	
	\vspace{.1cm}
	
	We begin by introducing fundamental concepts related to the fractional Sobolev spaces and the well-posedness of solutions to the non-local parabolic equations.
	
	\subsection{Preliminaries: Fractional Laplacian and fractional Sobolev space}	
	Let $0< a< 1$. The fractional Laplacian operator $(-\Delta)^a$ is defined over the space of Schwartz class functions $\mathcal{S}(\mathbb{R}^N)$ as
	\[
	(-\Delta)^a u(x):= \mathcal{F}^{-1}\left\{|\xi|^{2a}\Hat{u}(\xi)\right\}, \quad \forall \, x\in\mathbb{R}^N,
	\]
	where $\Hat{.}$ and $\mathcal{F}^{-1}$ denotes the Fourier and the inverse Fourier transformations respectively. There are many equivalent definitions of fractional Laplacian [see \cite{kwasnicki2017ten}]. One of such equivalent definition we use throughout this article is  given by the Cauchy principle value: for $0< a< 1$
	\[
	(-\Delta)^a u(x)= C_{N,a} \ \mbox{p.v} \ \int_{\mathbb{R}^N} \frac{u(x)-u(y)}{| x-y|^{N+2a}} \, \rm{d}y, \qquad \forall \, x\in\mathbb{R}^N, 
	\] 
	where $\displaystyle{C_{N,a}=\frac{4^a\Gamma\left(\frac{N}{2}+a\right)}{\pi^{\frac{N}{2}}\Gamma(-a)}}$. Throughout the article, $\Gamma$ stands for the usual Gamma function. We define the following Sobolev space: for $s\in(0,1)$
	\[
	H^s(\mathbb{R}^N):= \{ u\in\mathbb{S}'(\mathbb{R}^N): \langle \xi\rangle^{s}\Hat{u}\in \mathrm{L}^2(\mathbb{R}^N)\},
	\]
	where $\displaystyle{\langle \xi\rangle=(1+\Vert \xi\Vert_{2}^2)^{\frac{1}{2}}}$ and $\mathbb{S}'(\mathbb{R}^N)$ is the space of tempered distributions in $\mathbb{R}^N$, equipped with the norm
	\[
	\Vert u\Vert_{H^s(\mathbb{R}^N)}= \Vert \langle \xi\rangle^s \Hat{u}\Vert_{\mathrm{L}^2(\mathbb{R}^N)}.
	\]
	The fractional Laplacian extends as a bounded linear map 
	\[
	(-\Delta)^a: H^s(\mathbb{R}^N) \to H^{s-2a}(\mathbb{R}^N).
	\]
	We also introduce few spaces that are essential for describing the solution spaces and various regularity results throughout the article.
	\begin{align*}
		H^s(\Omega):=& \{u= v|_{\Omega}: v \in H^s(\mathbb{R}^N)\}\\
		C^s(\Omega):=& \{u\in C(\Omega): |u(x)-u(y)|\leq M|x-y|^s, \, \forall \, x,y\in \Omega, M\geq0\}\\
		C^s_c(\mathbb{R}^N):=& \{ u\in C^s(\mathbb{R}^N): \, \mbox{supp}\, u \ \mbox{compactly contained in}\ \mathbb{R}^N\},
	\end{align*}
	where $\Omega$ is a bounded $C^{1,1}$ domain and $s\in(0,1)$. We denote $v|_{\Omega}$ as the restriction of the function $v$ in $\Omega$. $H^s(\Omega)$ is a Sobolev space equipped with the norm 
	\[
	\Vert u\Vert_{H^s(\Omega)}= \Vert v\Vert_{H^s(\mathbb{R}^N)}= \Vert \langle \xi\rangle^s \Hat{v}\Vert_{\mathrm{L}^2(\mathbb{R}^N)}.
	\] 
	The definition of the space $C^s(\Omega)$ can be extended to the set when $\Omega$ is closed i.e., $\Omega=\overline{\Omega}$ or $\Omega=\mathbb{R}^N$. We define the following seminorm associated with the space $C^s(\Omega)$:
	\[
	[u]_{C^s(\Omega)}:= \inf_{M\geq 0} \left \{M: |u(x)-u(y)|\leq M |x-y|^s, \, \forall \, x,y\in\Omega \right\}.
	\]
	The spaces $C_c^s(\mathbb{R}^N)$ and $C^s(\overline{\Omega})$ $(\Omega$ is bounded$)$  can be endowed with the following norms respectively
	\begin{align*}
		\Vert u\Vert_{C^s_c(\mathbb{R}^N)}:= &\sup\limits_{x\in \mathbb{R}^N} |u(x)|+ [u]_{C^s(\mathbb{R}^N)},\\
		\Vert u\Vert_{C^s(\overline{\Omega})}:= &\sup\limits_{x\in\overline{\Omega}} |u(x)|+ [u]_{C^s(\Omega)}.
	\end{align*}
	Note that $C^s(\overline{\Omega})$ is a Banach space with respect to the norm described above.
	\newline
	The definition of $C^s(\Omega)$ $($or $C_c^s(\mathbb{R}^N))$ can be extended to any $\beta>0$. We can express $\beta=k+s$, where $k\in\mathbb{N}\cup\{0\}$ and $s\in(0,1)$. The space $C^{\beta}(\Omega)$ is the following
	\[
	C^{\beta}(\Omega):=\left\{ D^{\alpha}u\in C(\Omega): |D^{k}u(x)-D^ku(y)|\leq M|x-y|^s, \, \forall \, x,y\in\Omega, M\geq0\right\}
	\]
	where $\displaystyle{\alpha=(\alpha_1,\cdots,\alpha_N)\in(\mathbb{N}\cup\{0\})^N}$ and $\displaystyle{\sum\limits_{i=1}^{N}\alpha_i \leq k}$ and $D^{\alpha}$ is the classical derivative of order $\alpha$. Next we proceed to define $\mu$ transmission spaces.
	\subsection{$\mu$ transmission spaces, well posedness of solution}
	All the notations and definitions here are well explained in \cite{grubb2014local}, \cite{grubb2015fractional}, \cite{grubb2016regularity}, \cite{grubb2016integration}, \cite{grubb2018green}, \cite{Grubb2020ExactGF}. We only recall some important definitions and concepts here. 
	\newline
	We consider operators acting on the upper or lower half subspaces of $\mathbb{R}^N$ denoted by $\displaystyle{\mathbb{R}^N_{\pm}:=\{(x_1,\cdots,x_N)\in\mathbb{R}^N: x_N\gtrless 0\}}$. Similarly with respect to the boundary $\partial\Omega$ we divide a $C^{1,1}$ domain $\Omega$ into two parts. One part is the domain $\Omega$ itself and the other is the complement of it's closure $\displaystyle{\mathbb{R}^N\setminus\overline{\Omega}}$. We denote $r^{\pm}$ as the restriction operator from $\mathbb{R}^{N}$ to  $\displaystyle{\mathbb{R}^N_{\pm}}$ $\Big($or from $\mathbb{R}^{N}$ to $\Omega$ respectively $\displaystyle{\mathbb{R}^N\setminus\overline{\Omega}}\Big)$. The extension by zero from $\displaystyle{\mathbb{R}^N\setminus\overline{\Omega}}$ to $\mathbb{R}^N$ $\Big($or from $\Omega$ respectively $\displaystyle{\mathbb{R}^N\setminus\overline{\Omega}}$ to $\mathbb{R}^{N}\Big)$ is denoted by $e^{\pm}$.
	\newline
	Let us denote a function $\tilde{d}$ which is of the form $d(x):=\text{dist}(x,\partial\Omega)$ near boundary and extended smoothly into a positive function in $\Omega$. We define the space
	\begin{align}\label{density in transmission space}
		\mathcal{E}_{\mu}\left(\overline{\Omega}\right):= e^{+}\left\{u(x)=\tilde{d}^{\mu}v(x): v\in C^{\infty}\left(\overline{\Omega}\right)\right\},
	\end{align}
	for Re$(\mu)>-1$. For general $\mu$, we refer readers to \cite{grubb2015fractional}, \cite{seeley1971norms}. In this article we consider real valued $\mu$ only. Next we define pseudodifferential operator ($\psi$do) $\mathcal{P}$ on $\mathbb{R}^N$, which is defined from a symbol $p(x,\xi)$ on $\displaystyle{\mathbb{R}^N\times \mathbb{R}^N}$ in the following way:
	\begin{align*}
		\mathcal{P}u=p(x,D)u=\text{Op}((p(x,\xi))u= (2\pi)^{-N}\int e^{ix\dot\xi}p(x,\xi)\hat{u} \, \rm{d}\xi=\mathcal{F}_{\xi\to x}^{-1}((p(x,\xi)\hat{u}(\xi)),
	\end{align*}
	where $\hat{u}$ is the Fourier transformation of $u$. For the theory of calculus, we refer to books and notes like \cite{hormander1963linear}, \cite{hormander1966seminar}, \cite{taylor2006pseudo}, \cite{grubb2008distributions}. 
	We define the symbol space $S^m_{1,0}\left(\mathbb{R}^N\times \mathbb{R}^N\right)$ as follows:
	\begin{equation*}
		S^m_{1,0}:=
		\left\{
		\begin{aligned}
			p(x,\xi)\in C^{\infty}\left(\mathbb{R}^N\times \mathbb{R}^N\right): &\ \vert \partial_{x}^{\beta}\partial_{\xi}^{\alpha} p(x,\xi)\vert \sim O\left(\langle\xi\rangle^{m-|\alpha|}\right),\\
			&  \forall \alpha,\beta\in\mathbb{N}\cup\{0\}, \ \mbox{for some}\ m\in\mathbb{C}
		\end{aligned}
		\right \}
	\end{equation*}
	where $\displaystyle{\langle\xi\rangle:=\left(1+\Vert \xi\Vert^2_2\right)^{\frac{1}{2}}}$ and $O$ is the big-O function captures the order. The exponent $m$ is called the order of $\mathcal{P}($and $p)$. Note that the operator $\left((-\Delta)^{a}+q\right)$ is a $\psi$do operator of order 2a. Depending on the context, the operator $\mathcal{P}_+=r^{+}\mathcal{P}e^{+}$ denotes the truncation of $\mathcal{P}$ to $\mathbb{R}^{N}_{+}$ or to $\Omega$. We define the following $\mathrm{L}^2$ spaces:
	\begin{align*}
		&H^s\left(\mathbb{R}^N\right):= \left\{ u\in\mathcal{S}'\left(\mathbb{R}^N\right)|\mathcal{F}^{-1}\left(\langle\xi\rangle^s\hat{u}\right)\in \mathrm{L}^2\left(\mathbb{R}^N\right)\right\},
		\\
		&\dot{H}^s\left(\overline{\Omega}\right):= \left\{ u\in H^s\left(\mathbb{R}^N\right)| \ \mbox{supp}\ u\subset \overline{\Omega}\right\},
		\\
		&\overline{H}^s(\Omega):= \left\{ u\in \mathcal{D}'(\Omega)| u=r^{+}U \ \mbox{for a}\  U\in H^s\left(\mathbb{R}^N\right)\right\}.
	\end{align*}
	Next we define order reducing operator on $\mathbb{R}^N$: for $t\in\mathbb{R}$
	\[
	\Xi^{t}_{\pm}:= Op\left(\chi^t_{\pm}\right), \quad \chi^t_{\pm}:=\left( \langle \xi'\rangle\pm i\xi_N\right)^t \ \mbox{where} \ \xi=(\xi',\xi_N), \xi':=(\xi_1,\dots,\xi_{N-1}).
	\]
	These symbols extend analytically in $\xi_N$ to Im $\displaystyle{\xi_N\gtrless 0}$. Hence by Paley-Wiener theorem $\displaystyle{\Xi^{t}_{\pm}}$ preserves support in $\displaystyle{\mathbb{R}^N_{\pm}}$. Similarly one can defined $\displaystyle{\Lambda^t_{\pm}}$ adapted to the situation for a smooth domain. For details readers are refereed to \cite{hormander1963linear}, \cite{hormander1966seminar}, \cite{taylor2006pseudo}, \cite{taylor1991pseudodifferential} \cite{grubb2008distributions}. Here we mention some important homeomorphism results regarding this operators: for all $s>0$
	\begin{align*}
		& \Xi^{t}_{\pm}: H^s\left(\mathbb{R}^N\right) \xrightarrow{\sim} H^{s-t}\left(\mathbb{R}^N\right), \ \Lambda^{t}_{\pm}: H^s\left(\mathbb{R}^N\right) \xrightarrow{\sim} H^{s-t}\left(\mathbb{R}^N\right),\\
		& \Xi^{t}_{\pm}: \dot{H}^s\left( \overline{\mathbb{R}^{N}_{\pm}}\right) \xrightarrow{\sim}\dot{H}^{s-t}\left( \overline{\mathbb{R}^{N}_{\pm}}\right), \ r^{+}\Xi^{t}_{\pm} e^{+}: \overline{H}^s\left( {\mathbb{R}^{N}_{\pm}}\right) \xrightarrow{\sim}\overline{H}^{s-t}\left( {\mathbb{R}^{N}_{\pm}}\right),\\
		& \Lambda^{t}_{+}: \dot{H}^s\left( \overline{\Omega}\right) \xrightarrow{\sim}\dot{H}^{s-t}\left( \overline{\Omega}\right), \ r^{+}\Lambda^{t}_{-} e^{+}: \overline{H}^s\left( {\Omega}\right) \xrightarrow{\sim}\overline{H}^{s-t}\left( {\Omega}\right).
	\end{align*}
	We recall $\mu$-transmission spaces with Re $\mu>-1$ as introduced by H\"ormander \cite{hormander1966seminar} and redefined in \cite{grubb2015fractional}:
	\begin{align*}
		H^{\mu(s)}\left(\overline{\mathbb{R}_{+}^N}\right):= &\Xi_{+}^{-\mu} e^{+}\overline{H}^{s-\mu}\left(\mathbb{R}_{+}^N\right), \quad s>\mu-\frac{1}{2}\\
		H^{\mu(s)}\left(\overline{\Omega}\right):= & \Lambda_{+}^{(-\mu)} e^{+} \overline{H}^{s-\mu}(\Omega), \ \ \quad s>\mu-\frac{1}{2}.
	\end{align*}
	The space $\displaystyle{{\mathcal{E}_{\mu}\left(\overline{\Omega}\right)}}$ is dense in $\displaystyle{H^{\mu(s)}\left(\overline{\Omega}\right)}$ i.e., [see  \cite{hormander1966seminar}, \cite{grubb2015fractional}]
	\[
	\cap_{s} H^{\mu(s)}\left(\overline{\Omega}\right) = \mathcal{E}_{\mu}\left(\overline{\Omega}\right).
	\]
	Also we have the following characterization [see \cite{grubb2015fractional}]:
	\begin{equation*}
		H^{\mu(s)}\left(\overline{\Omega}\right)
		\left\{
		\begin{aligned}
			&= \dot{H}^s \left(\overline{\Omega}\right) \ \ \ \mbox{if} \ |s-\mu|<\frac{1}{2},\\
			& \subset \dot{H}^s \left(\overline{\Omega}\right)+ e^{+}\tilde{d}^{\mu}\overline{H}^{s-\mu}(\Omega), \ \ \mbox{if}\ s>\mu+\frac{1}{2}, s-\mu-\frac{1}{2}\not \in \mathbb{N}.
		\end{aligned}
		\right .
	\end{equation*}
	Some further characterization can be found in \cite{grubb2015fractional}, \cite{grubb2014local} \cite{ghosh2020calderon}. These characterizations are stated in the lemma below:
	\begin{Lem}\label{Further chracterization of mu transmission space}
		Let $\Omega$ be a smooth bounded domain. The following properties hold:
		\begin{itemize}
			\item[a:] For all $s>\mu-\frac 12$, $H^{\mu(s)}(\overline{\Omega})\subset \overline{H}^{\mu-\frac12}(\Omega)$ with continuous inclusion.
			\item[b:] 
			$H^{\mu(s)}(\overline{\Omega})\subset\dot{H}^{\mu+\frac 12-\epsilon}(\overline{\Omega})$, for all $s\geq\mu+\frac 12$ and $\epsilon>0$.
			\item[c:] $H^{\mu(s)}(\overline{\Omega})=\dot{H}^{s}(\overline{\Omega})$ if $\displaystyle{s\in\left( \mu-\frac12, \mu+\frac12\right)}$.
			\item[d:] Let, $q\in C_c^{\infty}(\Omega)$. The operator $\displaystyle{r^+\big((-\Delta)^{\mu}+q\big)}$ is a homeomorphism  from $H^{\mu(s)}(\overline{\Omega})$ onto $H^{s-2\mu}(\Omega)$.
			\item[e:] $\dot{H}^{s}(\overline{\Omega}) \subset H^{\mu(s)}(\overline{\Omega}) \subset H^{s}_{loc}(\Omega)$ with continuous inclusions; i.e., multiplication by any $\chi\in C_c^{\infty}(\Omega)$ is bounded $H^{\mu(s)}(\overline{\Omega}) \to H^s(\Omega)$.
		\end{itemize}
	\end{Lem}
	The points (a), (d), (e) of the above theorem can be found in \cite[Lemma 7.1]{ghosh2020calderon}. The point (b) can be found in \cite{grubb2023resolvents}, whereas point (c) can be found in \cite{grubb2023resolvents} and point (c) can be fond in \cite[Theorem 4.5]{abels2023fractional}. The domain of fractional  Laplacian with homogeneous Dirichlet boundary condition can be characterize as follows:
	\begin{Lem}[\cite{claus2020realization}, \cite{grubb2018regularity}]\label{domain} 
		Let $\Omega\subset\mathbb{R}^N$ ($N\ge 1$) be a bounded open set with smooth boundary and $a\in(0,1)$. Let $D((-\Delta)^a)$ be the domain of the operator $(-\Delta)^a$ with homogeneous Dirichlet boundary condition outside $\Omega$.  Then, the following assertions holds.
		\begin{itemize}
			\item[(a)] If $a\in(0,1/2)$, then $D((-\Delta)^a)= \dot{H}^{2a}(\overline{\Omega})$.
			\item[(b)] If $a=1/2$, then $D((-\Delta)^a)\subset \dot{H}^{1-\varepsilon}(\overline{\Omega})$ for all $\varepsilon\in(0,1)$.
			\item[(c)] If $a\in(1/2,1)$, then $D((-\Delta)^a)\subset \dot{H}^{a+\frac 12-\epsilon}(\overline{\Omega})$.
		\end{itemize}
	\end{Lem}	   
	We now state a wellposedness and regularity result corresponding to perturbed heat equation with singular boundary data in $\mu$-transmission spaces. This result is useful to obtain an integration by parts formula. Proof of this non-local parabolic wellposedness and regularity result can be found in \cite[Theorem 9]{das2025boundarycontrolcalderontype}, \cite{grubb2023resolvents}.
	\begin{Thm}[{\cite[Theorem 9]{das2025boundarycontrolcalderontype}}, \cite{grubb2023resolvents}]\label{s-transmission regularity, boundary data}
		Let $f\in \mathrm{L}^2\left((0,T); H^{a+\frac{1}{2}}(\partial\Omega)\right)\cap \overline{H}^1\left((0,T); H^{\epsilon}(\partial\Omega)\right)$, for some $\epsilon>0$ and let $f(0,\cdot)=0$. Consider the following equation
		\begin{equation}\label{parabolic singular boundary data}
			\left\{
			\begin{aligned}
				\partial_t u+(-\Delta)^a u+ qu=&0 \qquad \qquad  \mbox{in}\ (0,T)\times\Omega\\
				u=&0 \qquad \qquad  \mbox{in} \ (0,T)\times\mathbb{R}^N\setminus\overline{\Omega}\\
				\lim\limits_{\substack{x\to y\\ y\in\partial\Omega}}\frac{u}{d^{a-1}}(y)=& f \qquad \qquad \mbox{on} \ (0,T)\times\partial\Omega\\
				u(0,\cdot)=&0 \qquad \qquad  \mbox{in} \ \Omega,
			\end{aligned}
			\right .
		\end{equation}
		where $q\in C_c^{\infty}(\Omega)$ is not an eigenvalue to \eqref{parabolic singular boundary data} corresponding to $f\equiv 0$. Then there exists a unique solution $u$, having the following regularity
		\[
		u \in \mathrm{L}^2\left((0,T); H^{(a-1)(2a)}(\overline{\Omega})\right) \cap \overline{H}^1\left((0,T); \mathrm{L}^2(\Omega)\right).
		\]
	\end{Thm}	
	Next we move onto defining traces of the transmission spaces.
	\subsection{Boundary values of $\mu$ transmission spaces, Dirichlet trace}\label{trace}
	The definitions and notations used in this sub-section are well explained in \cite{grubb2015fractional}. The author in \cite{grubb2015fractional}, explained that the $\mu$ transmission spaces $H^{\mu(s)}$ admits a special definition of $\mu$ boundary values, which characterize how the elements of the $\mu$ transmission spaces decays towards the boundary. Let $M$ be a positive integer. We recall the space $\mathcal{E}_{\mu}$ as defined in \eqref{density in transmission space}. Let us introduce the natural canonical mapping 
	\begin{align}\label{canonical map}
		\rho_{\mu,M}: \mathcal{E}_{\mu}\to \mathcal{E}_{\mu}/\mathcal{E}_{\mu+M},
	\end{align}
	where one can represent the space $\displaystyle{\frac{\mathcal{E}_{\mu}}{\mathcal{E}_{\mu+M}}}$ as the space of sections of a trivial bundle and introduce norms in it. Let us choose a smooth function $\tilde{d}$ on $\overline{\Omega}$ which is equal to the distance from the boundary sufficiently close to the boundary and remains positive in $\Omega$. We set
	\[
	I^{\mu}(x):= \frac{\tilde{d}^{\mu}(x)}{\Gamma(\mu+1)} \ \mbox{in}\ \overline{\Omega} \ \ \mbox{and}\ I^{\mu}\equiv 0 \ \mbox{in} \partial\Omega,
	\]
	when $(\mu)$>-1. This definition can be uniquely extended modulo $C_0^{\infty}(\Omega)$ to arbitrary values of $\mu$ so that $\partial_N I^{\mu}=I^{\mu-1}$, where $\partial_N$ is the derivative with respect to the $x_N$ variable. Note that fact that the line $x_N$ is perpendicular to the boundary $\partial\Omega$. Hence, it can be viewed as the differentiation along the geodesics perpendicular to the boundary. Our definition on $\mathcal{E}_{\mu}$ \eqref{density in transmission space}, clearly indicates that every class in $\displaystyle{\mathcal{E}_{\mu}/\mathcal{E}_{\mu+1}}$ contains an element of the form $I^{\mu}f$ where $f\in C^{\infty}(\overline{\Omega)}$. Furthermore, in the quotient space that element is equivalent to $0$ if and only if $f=0$ on the boundary $\partial\mathbb{R}_+^N$. Repeated application of the previous fact yields the following representation of $u\in\mathcal{E}_{\mu}$:
	\[
	u= u_0 I^{\mu}+u_1I^{\mu+1}+\cdots+u_{M-1}I^{\mu+M-1}+v,
	\] 
	where the functions $u_i\in C^{\infty}(\overline{\Omega})$ and constant along the normal closed to the boundary $\partial\Omega$. So, the traces of $u$ denoted by the following notation
	\[
	\gamma_{j}^{\mu}u= u_j|_{\partial\Omega}.
	\]
	Note that the fact that closed to the boundary, $I^{\mu}$ is equal to $\displaystyle{\frac{d^{\mu}}{\Gamma(\mu+1)}}$. Furthermore, note that:
	\begin{align*}
		\gamma_{j}^{\mu}u=& \gamma_{0}^{\mu+j}u, \quad \mbox{when} \ u\in\mathcal{E}_{\mu+j},\\
		\gamma_{0}^{\mu}u=& \Gamma(\mu+1)\gamma_0 d^{-\mu}u, \quad \mbox{when} \ u\in\mathcal{E}_{\mu}, \ \mbox{with} \ \mu>-1.
	\end{align*}
	where $\gamma_0$ is the projection of a function on the boundary $\partial\Omega$. We can extend the boundary trace for the functions in $H^{\mu(s)}(\overline{\Omega})$. For a function
	$u\in H^{\mu(s)}(\overline{\Omega})$, the boundary term $\frac{u}{d^{\mu}}$ is well defined and lies in the space $H^{s-\mu-\frac12}(\partial\Omega)$ [see \cite{grubb2017fractional}, \cite{grubb2018green}]. The boundary trace $\frac{u}{d^{\mu}}$ for the function $u\in H^{\mu(s)}(\overline{\Omega})$, we call it as \emph{Dirichlet trace}. The following result can be found in \cite{grubb2015fractional}:
	\begin{Thm}[\cite{grubb2015fractional}]\label{dirichlet kernel in mu-transmission space}
		Let $s>\mu+M-\frac{1}{2}$. The canonical map $\rho_{\mu,M}$ in \eqref{canonical map} can be extended by continuity (still denoted by $\rho_{\mu,M}$):
		\[
		\rho_{\mu,M}: H^{\mu(s)}(\overline{\Omega})\to \prod_{0\leq j<M} H^{s-\mu-j-\frac12}(\partial\Omega).
		\]
		Furthermore, the map is surjective and with kernel $H^{(\mu+M)(s)}(\overline{\Omega})$. In other words, the canonical map $\rho_{\mu,M}$ defines a homeomorphism of $\displaystyle{H^{\mu(s)}(\overline{\Omega})/H^{(\mu+M)(s)}(\overline{\Omega})}$ onto $\displaystyle{\prod_{0\leq j<M} H^{s-\mu-j-\frac12}(\partial\Omega)}$.
	\end{Thm}
	The theorem implies that if a function $u\in H^{\mu(s)}(\overline{\Omega})$ has zero Dirichlet trace, with $s>\mu+M-\frac12$, then $u\in H^{(\mu+1)(s)}(\overline{\Omega})$.  Similarly, we can define the Neumann trace also for the transmission spaces $H^{(a-1)(2a)}(\overline{\Omega})$. 
	\subsection{Dirichlet to Neumann (DN) map}\label{Definition of DN map}
	Consider $g\in H^{s-a+\frac{1}{2}}(\overline{\Omega})$. Let $v$ satisfy
	\begin{equation}\label{singular boundary elliptic equation}
		\left \{
		\begin{aligned}
			(-\Delta)^a v+qv=& 0 \qquad \mbox{in} \ \Omega\\
			v=& 0 \qquad \mbox{in} \ \mathbb{R}^N\setminus\overline{\Omega}\\
			\frac{v}{d^{a-1}}= & g \qquad \mbox{on} \ \partial{\Omega},
		\end{aligned}
		\right .
	\end{equation} 
	where $q\in C_c^{\infty}(\Omega)$. We can divide $v$ in two parts i.e. $v=\tilde{v}_1+\tilde{v}_2$, where $\tilde{v}_1$ satisfies
	\begin{equation*}
		\left \{
		\begin{aligned}
			(-\Delta)^a \tilde{v}_1=& 0 \qquad \mbox{in} \ \Omega\\
			\tilde{v}_1=& 0 \qquad \mbox{in} \ \mathbb{R}^N\setminus\overline{\Omega}\\
			\frac{\tilde{v}_1}{d^{a-1}}= & g \qquad \mbox{on} \ \partial{\Omega}.
		\end{aligned}
		\right .
	\end{equation*}
	and $\tilde{v}_2$ satisfies 
	\begin{equation*}
		\left \{
		\begin{aligned}
			(-\Delta)^a \tilde{v}_2+q\tilde{v}_2=& -q\tilde{v}_1 \qquad \mbox{in} \ \Omega\\
			\tilde{v}_2=& 0 \qquad \mbox{in} \ \mathbb{R}^N\setminus\Omega.
		\end{aligned}
		\right .
	\end{equation*}
	Thanks to \cite{grubb2023resolvents}, \cite{grubb2015fractional}, we have that $\tilde{v}_1\in H^{(a-1)(s)}(\overline{\Omega})$. Furthermore the \emph{Dirichlet to Neumann (DN) map}
	\[
	\gamma_1^{a-1}: H^{s-a+\frac{1}{2}}(\partial\Omega) \to H^{s-a-\frac{1}{2}}(\partial\Omega)
	\]
	is a well defined bounded linear map [see \cite{grubb2017fractional}, \cite{grubb2018green}], given by 
	\[
	\gamma_1^{a-1} (g)= \partial_{\nu} \frac{\tilde{v}_1}{d^{a-1}}\big|_{\partial\Omega}.
	\]
	Moreover, thanks to \cite{grubb2023resolvents}, \cite{grubb2015fractional}, we have that $\tilde{v}_2\in H^{a(s)}(\overline{\Omega})$. The \emph{Dirichlet trace map}
	\[
	\gamma_0^a: H^{a(s)}(\overline{\Omega})\to H^{s-a-\frac{1}{2}}(\partial\Omega)
	\] 
	is a well defined bounded linear map [see \cite{grubb2017fractional}, \cite{grubb2018green}], given by 
	\[
	\gamma_0^a(\tilde{v}_2)=\frac{\tilde{v}_2}{d^a}\Big|_{\partial\Omega}.
	\] 
	We define the following \emph{Dirichlet to Neumann (DN) map} corresponding to \eqref{singular boundary elliptic equation}
	\[
	\Lambda: H^{s-a+\frac{1}{2}}(\partial\Omega)\to H^{s-a-\frac{1}{2}}(\partial\Omega)
	\]
	in the following way
	\[
	\Lambda(g)= \gamma_1^{a-1}(g)+\gamma_0^a(\tilde{v}_2).
	\]
	This is well defined continuous linear map [see \cite{grubb2017fractional}, \cite{grubb2018green}].
	This Neumann data is conventionally expressed as 
	\begin{align}\label{elliptic DN map}
		\Lambda(g)\equiv\partial_{\nu}\left(\frac{v}{d^{a-1}}\right)\Big|_{\partial\Omega}.
	\end{align}
	Furthermore, the following integration by parts formula holds \cite{grubb2017fractional}, \cite{grubb2018green}: If $\displaystyle{v_1,v_2\in H^{(a-1)(s)}(\overline{\Omega})}$ with $\displaystyle{s>a+\frac{1}{2}}$, then
	\begin{align}\label{integration by parts}
		\int_{\Omega}v_1(-\Delta)^a v_2 -v_1(-\Delta)^a v_2 = \Gamma(a)\Gamma(a+1) \int_{\partial\Omega} \partial_{\nu} \left(\frac{v_1}{d^{a-1}}\right) \frac{v_2}{d^{a-1}}- \partial_{\nu} \left(\frac{v_1}{d^{a-1}}\right) \frac{v_2}{d^{a-1}}.
	\end{align}
	Furthermore, if $v_1\in H^{(a-1)(s)}(\overline{\Omega})$ and $v_2\in H^{a(s)}(\overline{\Omega})$, then
	\begin{align}\label{integration by parts dirichlt}
		\int_{\Omega}v_2(-\Delta)^a v_1 -v_1(-\Delta)^a v_2 = -\Gamma(a)\Gamma(a+1) \int_{\Omega} \frac{v_1}{d^{a-1}}\frac{v_2}{d^a}.
	\end{align}
	Here the boundary contributions are completely local. The formula was first proved for fractional laplacian in \cite{abatangelo2013large} and later generalized for any classical pseudodifferential operator of order $2a$ in \cite{grubb2016integration}, \cite{Grubb2020ExactGF}.\\
	The \emph{Dirichlet to Neumann (DN) map} can also be extended to the parabolic case in a similar manner. For the parabolic case we denote the \emph{Dirichlet to Neumann (DN) map} in the following manner
	\[
	\Lambda: \mathrm{L}^2((0,T);  H^{s-a+\frac{1}{2}}(\partial\Omega))\to \mathrm{L}^2((0,T); H^{s-a-\frac{1}{2}}(\partial\Omega)),
	\]
	given by
	\begin{align}\label{DN map}
		\Lambda (g)= \partial_{\nu}\left(\frac{u}{d^{a-1}}\right)\Big|_{\partial\Omega}, \quad \mbox{for a.e.}\ t>0,
	\end{align}
	where $u$ satisfies \eqref{parabolic singular boundary data}. This is also a well defined continuous linear map [see \cite{Grubb2020ExactGF}, \cite{grubb2017fractional}, \cite{grubb2018green}]. Hence, there exists a positive constant $C_{DN}$ such that
	\begin{align*}
		\Vert \Lambda(g)\Vert_{\mathrm{L}^2((0,T); H^{s-a-\frac12}(\partial\Omega))}=&\left\Vert \partial_{\nu}\left(\frac{u}{d^{a-1}}\right)\Big|_{\partial\Omega} \right\Vert_{\mathrm{L}^2((0,T); H^{s-a-\frac12}(\partial\Omega))}\\
		\leq& C_{DN} \Vert g\Vert_{\mathrm{L}^2((0,T); H^{s-a+\frac12}(\partial\Omega))}. \nonumber
	\end{align*}
	In particular, the constant $C_{DN}$ doesn't depends on the parameter `$T$'. It can be easily verified by taking the transformation $t|\to \frac{t}{T}$. Furthermore if $s-a-\frac12>0$, the continuous inclusion $H^{s-a-\frac12}\hookrightarrow \mathrm{L}^2$ yields that 
	\begin{align}\label{DN continuity}
		\Vert \Lambda(g)\Vert_{\mathrm{L}^2((0,T); \mathrm{L}^2(\partial\Omega))}=\left\Vert \partial_{\nu}\left(\frac{u}{d^{a-1}}\right)\Big|_{\partial\Omega} \right\Vert_{\mathrm{L}^2((0,T); \mathrm{L}^2(\partial\Omega))}
		\leq C_{DN} \Vert g\Vert_{\mathrm{L}^2((0,T); H^{s-a+\frac12}(\partial\Omega))}.
	\end{align}	
	In the next section, we demonstrate how the fractional Borg-Levinson Criteria \eqref{Borg Levinson fractional criteria} lead to the equality of \emph{Dirichlet to Neumann (DN) maps} for the fractional elliptic case.
	
	\section{Equality of DN maps corresponding to two different potentials} \label{equality of dn maps}
	
	Here we consider the fractional version of Borg-Levinson problem. 
	Let $\lambda\in \mathbb{C}\setminus\{\lambda_n\}_{n\in\mathbb{N}}$. Consider the following equation: for $i=1,2$
	\begin{equation}\label{equation in complex variable outside spectrum}
		\left \{
		\begin{aligned}
			(-\Delta)^a v_{\lambda,i}+q_iv_{\lambda,i}=& \lambda v_{\lambda,i} \qquad \mbox{in} \ \Omega\\
			v_{\lambda,i}=& 0 \qquad \quad \ \mbox{in} \ \mathbb{R}^N\setminus\overline{\Omega}\\		\frac{v_{\lambda,i}}{d^{a-1}}=& F \qquad \quad  \mbox{on}\ \partial\Omega,  
		\end{aligned}
		\right .
	\end{equation}
	where $F\in C_c^{\infty}(\Omega)$. In particular the solution $v_{\lambda,i}$ exists and have the following regularity \cite{grubb2014local}
	\[
	v_{\lambda,i}\in H^{(a-1)(2a)}(\overline{\Omega})\cap e^+d(x)^aC^{\infty}(\overline{\Omega})(\equiv \mathcal{E}(\overline{\Omega}))
	\]
	The \emph{Dirichlet to Neumann map} is well defined as in \eqref{elliptic DN map} and the corresponding Dirichlet to Neumann boundary data for \eqref{equation in complex variable outside spectrum} is $\displaystyle{\partial_{\nu}\left(\frac{v_{\lambda,i}}{d^{a-1}}\right)\big|_{\partial\Omega}}\in H^{a-\frac{1}{2}}(\partial\Omega)$. Furthermore, as $v\in\mathrm{L}^2(\Omega)$, it admits the following representation:
	\[
	v_{\lambda,i}= \sum\limits_{n\in\mathbb{N}} \langle v_{\lambda,i},\phi_{n,q_i}\rangle_{\mathrm{L}^2(\Omega)}\phi_{n,q_i}, \quad \mbox{for}\ i=1,2,
	\]
	where $\langle,\rangle$ is the classical inner product in $\mathrm{L}^2$. Here we use the fact that $\displaystyle{\{\phi_{n,q_i}\}\big|_{n\in\mathbb{N}}}$  forms an orthonormal basis in $\mathrm{L}^2(\Omega)$ for $i=1,2$. As $a\in[\frac12,1)$, we use integration by parts formula as in \eqref{integration by parts dirichlt}. We have that
	\begin{align*}
		-\Gamma(a)\Gamma(a+1)\int_{\partial\Omega} F \frac{\phi_{n,q_i}}{d^{a}}\, \rm{d}\sigma= &\int_{\Omega} (-\Delta)^a v_{\lambda,i} \phi_{n,q_i}\, \rm{d}x- \int_{\Omega}(-\Delta)^a \phi_{n,q_i} v_{\lambda,i}\, \rm{d}x\\
		=\int_{\Omega} \phi_{n,q_i}\left((-\Delta)^a+q_i\right) v_{\lambda,i} \, \rm{d}x-& \int_{\Omega}v_{\lambda,i}\left((-\Delta)^a+q_i\right) \phi_{n,q_i}\, \rm{d}x\\
		= & \lambda\int_{\Omega} v_{\lambda,i}\phi_{n,q_i} \, \rm{d}x-\lambda_{n} \int_{\Omega} v_{\lambda,i}\phi_{n,q_i} \, \rm{d}x\\
		=& (\lambda-\lambda_n)\int_{\Omega} v_{\lambda,i}\phi_{n,q_i} \, \rm{d}x.
	\end{align*}
	Hence $v_{\lambda,i}$ can be expressed as 
	\begin{align}\label{representation of elliptic solution}
		v_{\lambda,i}= -\Gamma(a)\Gamma(a+1)\sum\limits_{n\in\mathbb{N}} \frac{1}{\lambda-\lambda_n}\left\langle F,\frac{\phi_{n,q_i}}{d^a}\right\rangle_{\mathrm{L}^2(\partial\Omega)} \phi_{n,q_i}, \quad \mbox{for} \ i=1,2,
	\end{align}
	with 
	\begin{align}\label{L2 norm elliptic solution}
		\sum_{n\in\mathbb{N}}  \frac{1}{| \lambda-\lambda_n|^2}\left\langle F,\frac{\phi_{n,q_i}}{d^a}\right\rangle_{\mathrm{L}^2(\partial\Omega)}^2 <\infty, \quad \forall \, \lambda\in \mathbb{C}\setminus \{\lambda_n\}|_{n\in\mathbb{N}}.
	\end{align}
	We now derive the following theorem, which establishes the equality of the boundary Neumann data for the difference of solutions.
	\begin{Prop}\label{equality of DN map for difference of solution}
		Let $\mu_1,\mu_2\in \mathbb{C}\setminus \{\lambda_n\}|_{n\in\mathbb{N}}$ and let $a\in\left(\frac12,1\right)$. For $i=1,2$, let $v_{\mu_1,i}, v_{\mu_2,i}$ denote the solutions to \eqref{equation in complex variable outside spectrum} corresponding to $\lambda=\mu_1$ and $\lambda=\mu_2$, respectively. Then, we have the following result: 
		\[
		\partial_{\nu}\left( \frac{v_{\mu_1,1}-v_{\mu_2,1}}{d^{a-1}}\right)= \partial_{\nu}\left( \frac{v_{\mu_1,2}-v_{\mu_2,2}}{d^{a-1}}\right).
		\]
	\end{Prop}
	\begin{proof}
		The difference of solutions corresponding to a particular potential $q_i$, can be expressed as: for $i=1,2$, 
		\[
		v_{\mu_1,i}-v_{\mu_2,i}= \Gamma(a)\Gamma(a+1)\sum\limits_{n\in\mathbb{N}}   \frac{\mu_1-\mu_2}{(\mu_1-\lambda_n)(\mu_2-\lambda_n)} \left\langle F, \frac{\phi_{n,q_i}}{d^{a-1}}\right\rangle_{\mathrm{L}^2(\partial\Omega)} \phi_{n,q_i}.
		\]
		Let us consider the following function, obtained as the projection onto the first $k$-dimensional subspace:
		\[
		V_{\mu_1,\mu_2,k,q_i}= \Gamma(a)\Gamma(a+1)\sum\limits_{n=1}^{k} \frac{\mu_1-\mu_2}{(\mu_1-\lambda_n)(\mu_2-\lambda_n)}\left\langle F,\frac{\phi_{n,q_i}}{d^{a-1}}\right\rangle \phi_{n,q_i}, \quad \mbox{for} \ i=1,2.
		\] 
		We further have 
		\[
		\frac{v_{\mu_1,i}-v_{\mu_2,i}}{d^{a-1}}\Big|_{\partial\Omega}\equiv \frac{V_{\mu_1,\mu_2,k,q_i}}{d^{a-1}}\Big|_{\partial\Omega}\equiv 0. 
		\]
		Hence $\displaystyle{\frac{V_{\mu_1,\mu_2,k,q_i}}{d^{a-1}}\Big|_{\partial\Omega}\xrightarrow{k\to+\infty} \frac{v_{\mu_1,i}-v_{\mu_2,i}}{d^{a-1}}\Big|_{\partial\Omega}}$ in $\mathrm{L}^2(\partial\Omega)$.
		Moreover, for $i=1,2$, the function $V_{\mu_1,\mu_2,k,q_i}$ satisfy the following equation:
		\begin{equation*}
			\left \{
			\begin{aligned}
				(-\Delta)^a V_{\mu_1,\mu_2,k,q_i}+&q_iV_{\mu_1,\mu_2,k,q_i} \\
				=& \Gamma(a)\Gamma(a+1)\sum\limits_{n=1}^{k} \frac{(\mu_1-\mu_2)\lambda_n}{(\mu_1-\lambda_n)(\mu_2-\lambda_n)}\left\langle F,\frac{\phi_{n,q_i}}{d^{a-1}}\right\rangle \phi_{n,q_i} \ \ \mbox{in} \ \Omega\\
				V_{\mu_1,\mu_2,k,q_i}=& 0 \qquad \qquad \qquad \qquad \qquad \qquad \qquad \qquad \qquad \qquad \quad  \mbox{in} \ \mathbb{R}^N\setminus\Omega.
			\end{aligned}
			\right .
		\end{equation*} 
		Here, the right hand side $\displaystyle{\Gamma(a)\Gamma(a+1)\sum\limits_{n=1}^{k} \frac{(\mu_1-\mu_2)\lambda_n}{(\mu_1-\lambda_n)(\mu_2-\lambda_n)}\left\langle F,\frac{\phi_{n,q_i}}{d^{a-1}}\right\rangle \phi_{n,q_i}}$ belongs to the space $\mathrm{L}^2(\Omega)$. Hence, for $i=1,2$, the solution $V_{\mu_1,\mu_2,k,q_i}$ belongs to the space $\displaystyle{H^{a(2a)}(\overline{\Omega})}$ and the quantity $\frac{V_{\mu_1,\mu_2,k,q_i}}{d^{a}}$ is well defined, lying in the space $H^{a-\frac12}(\partial\Omega)$ [see \cite{abatangelo2013large}, \cite{abatangelo2023singular}, \cite{grubb2014local},  \cite{ros2018higher}]. Thanks to the fact that $\displaystyle{\sup\limits_{n\in\mathbb{N}}\left| \frac{\lambda_n}{\mu_i-\lambda_n}\right|<+\infty}$ for $i=1,2$, we have that 
		\[
		\big((-\Delta)^a+q_i)\big)V_{\mu_1,\mu_2,k,q_i} \xrightarrow{k\to+\infty}\big((-\Delta)^a+q_i)\big)(v_{\mu_1,i}-v_{\mu_2,i}) \ \mbox{in} \, \mathrm{L}^2(\Omega), \ \mbox{for} \ i=1,2.
		\]
		We analyze the function 
		\[
		\frac{V_{\mu_1,\mu_2,k,q_i}}{d^a}= \Gamma(a)\Gamma(a+1)\sum\limits_{n=1}^{k} \frac{(\mu_1-\mu_2)}{(\mu_1-\lambda_n)(\mu_2-\lambda_n)}\left\langle F,\frac{\phi_{n,q_i}}{d^{a-1}}\right\rangle \frac{\phi_{k,q_i}}{d^a}.
		\]
		It yields
		\begin{align*}
			&\left\Vert 	\Gamma(a)\Gamma(a+1)\sum\limits_{n\in\mathbb{N}} \frac{(\mu_1-\mu_2)}{(\mu_1-\lambda_n)(\mu_2-\lambda_n)}\left\langle F,\frac{\phi_{n,q_i}}{d^{a-1}}\right\rangle \frac{\phi_{k,q_i}}{d^a}\right\Vert_{\mathrm{L}^2(\partial\Omega)}^2 \\
			& \leq C(\Omega) \left| \Gamma(a)\Gamma(a+1)(\mu_1-\mu_2)\right|^2 \sum\limits_{n\in\mathbb{N}} \left| \frac{1}{(\mu_1-\lambda_n)(\mu_2-\lambda_n)}\right|^2 \left|\left\langle F, \frac{\phi_{n,q_i}}{d^a}\right\rangle\right|^2 \left\Vert \frac{\phi_{n,q_i}}{d^a}\right\Vert^2_{\mathrm{L}^2(\partial\Omega)}\\
			& \leq  C(\Omega)\left| \Gamma(a)\Gamma(a+1)(\mu_1-\mu_2)\right|^2 \sup\limits_{n\in\mathbb{N}} \left| \frac{\lambda_n}{\mu_2-\lambda_n}\right|^2 \sum\limits_{n\in\mathbb{N}} \left| \frac{1}{(\mu_1-\lambda_n)}\right|^2 \left|\left\langle F, \frac{\phi_{n,q_i}}{d^a}\right\rangle\right|^2\\
			& < +\infty, \qquad \qquad \mbox{for} \ i=1,2.
		\end{align*}
		Here $C(\Omega)$ is a positive constant as in \eqref{eigen value regularity}. Hence, we have that, for $i=1,2$, 
		\begin{align*}
			\frac{V_{\mu_1,\mu_2,k,q_i}}{d^a} \xrightarrow{k\to +\infty}	\Gamma(a)\Gamma(a+1)\sum\limits_{n\in\mathbb{N}} \frac{(\mu_1-\mu_2)}{(\mu_1-\lambda_n)(\mu_2-\lambda_n)}\left\langle F,\frac{\phi_{n,q_i}}{d^{a-1}}\right\rangle \frac{\phi_{k,q_i}}{d^a} \ \mbox{in} \ \mathrm{L}^2(\partial\Omega).
		\end{align*}
		In a similar manner, we can show 
		\begin{align}\label{intermediate step to fractional convergence}
			& \left\Vert \Gamma(a)\Gamma(a+1)\sum\limits_{n\in\mathbb{N}} \frac{(\mu_1-\mu_2)\lambda_n}{(\mu_1-\lambda_n)(\mu_2-\lambda_n)}\left\langle F,\frac{\phi_{n,q_i}}{d^{a-1}}\right\rangle \phi_{n,q_i} \right\Vert_{\mathrm{L}^2(\Omega)}\\
			& \leq  \left| \Gamma(a)\Gamma(a+1)(\mu_1-\mu_2)\right|^2 \sum\limits_{n\in\mathbb{N}} \left| \frac{\lambda_n}{(\mu_1-\lambda_n)(\mu_2-\lambda_n)}\right|^2 \left|\left\langle F, \frac{\phi_{n,q_i}}{d^a}\right\rangle\right|^2 \left\Vert \phi_{n,q_i}\right\Vert^2_{\mathrm{L}^2(\Omega)}\nonumber \\
			& \leq  \left| \Gamma(a)\Gamma(a+1)(\mu_1-\mu_2)\right|^2 \sup\limits_{n\in\mathbb{N}} \left| \frac{\lambda_n}{\mu_2-\lambda_n}\right|^2 \sum\limits_{n\in\mathbb{N}} \left| \frac{1}{(\mu_1-\lambda_n)}\right|^2 \left|\left\langle F, \frac{\phi_{n,q_i}}{d^a}\right\rangle\right|^2\nonumber \\
			& < +\infty, \qquad \qquad \mbox{for} \ i=1,2. \nonumber
		\end{align}
		Equation \eqref{equation in complex variable outside spectrum} yields
		\begin{align*}
			(-\Delta)^av_{\mu_1,i}+q_iv_{\mu_1,i}-&(-\Delta)^av_{\mu_2,i}-q_iv_{\mu_2,i}=  \mu_1v_{\mu_1,i}-\mu_2v_{\mu_2,i}\\
			=& \Gamma(a)\Gamma(a+1)\sum\limits_{n\in\mathbb{N}} \frac{(\mu_1-\mu_2)\lambda_n}{(\mu_1-\lambda_n)(\mu_2-\lambda_n)}\left\langle F,\frac{\phi_{n,q_i}}{d^{a-1}}\right\rangle \phi_{n,q_i}.
		\end{align*}
		Hence, thanks to \eqref{intermediate step to fractional convergence}, we obtain
		\[
		\left((-\Delta)^a+q_i\right) V_{\mu_1,\mu_2,k,q_i} \xrightarrow{k\to+\infty} \left((-\Delta)^a+q_i\right)(v_{\mu_1,i}-v_{\mu_2,i})  \ \mbox{in} \  \mathrm{L}^2(\Omega), \ \mbox{for} \ i=1,2.
		\]
		We compile these convergence results below: for $i=1,2$
		\begin{equation}\label{four convergence}
			\left \{
			\begin{aligned} 
				V_{\mu_1,\mu_2,k,q_i} \xrightarrow{k\to+\infty} v_{\mu_1,i}-v_{\mu_2,i} \ \ \ \ \ \ \ \  \mbox{in} \ \mathrm{L}^2(\Omega)\\
				\left((-\Delta)^a+q_i\right) V_{\mu_1,\mu_2,k,q_i} \xrightarrow{k\to+\infty} \left((-\Delta)^a+q_i\right)(v_{\mu_1,i}-v_{\mu_2,i}) \ \ \mbox{in} \mathrm{L}^2(\Omega)\\
				\frac{V_{\mu_1,\mu_2,k,q_i}}{d^{a-1}}\Big|_{\partial\Omega}\xrightarrow{k\to+\infty} \frac{v_{\mu_1,i}-v_{\mu_2,i}}{d^{a-1}}\Big|_{\partial\Omega} \ \ \mbox{in} \mathrm{L}^2(\partial\Omega)  \\
				\frac{V_{\mu_1,\mu_2,k,q_i}}{d^a} \xrightarrow{k\to +\infty}	\Gamma(a)\Gamma(a+1)\sum\limits_{n\in\mathbb{N}} \frac{(\mu_1-\mu_2)}{(\mu_1-\lambda_n)(\mu_2-\lambda_n)}\left\langle F,\frac{\phi_{n,q_i}}{d^{a-1}}\right\rangle \frac{\phi_{k,q_i}}{d^a} \ \mbox{in} \ \mathrm{L}^2(\partial\Omega).
			\end{aligned}
			\right .
		\end{equation}
		Consider the equation
		\begin{equation}\label{auxiliary equation U}
			\left \{
			\begin{aligned}
				(-\Delta)^a U=& g \qquad \mbox{in} \ \Omega\\
				U=& 0 \qquad \mbox{in} \ \mathbb{R}^N\setminus\overline{\Omega}\\
				\frac{U}{d^{a-1}}=& f \qquad \mbox{on}\ \partial\Omega,
			\end{aligned}
			\right . 
		\end{equation}
		where $f,g\in C_c^{\infty}(\Omega)$. The solution $U$ exists and have the following regularity \cite{grubb2014local}
		\[
		U \in H^{(a-1)(2a)}(\overline{\Omega})\cap e^+d(x)^aC^{\infty}(\overline{\Omega})(\equiv \mathcal{E}(\overline{\Omega})).
		\]
		The \emph{Dirichlet to Neumann (DN) map} is well defined as in \eqref{elliptic DN map} and the corresponding Dirichlet to boundary Neumann data for \eqref{auxiliary equation U} is $\displaystyle{\partial_{\nu}\left(\frac{U}{d^{a-1}}\right)\big|_{\partial\Omega}}\in H^{a-\frac{1}{2}}(\partial\Omega)$. For $i=1,2$, using integration by parts formula \eqref{integration by parts}, we obtain
		\begin{align*}
			\langle (-\Delta)^a &V_{\mu_1,\mu_2,k,q_i}, U\rangle_{\mathrm{L}^2(\Omega)} - \langle (-\Delta)^a U, V_{\mu_1,\mu_2,k,q_i}\rangle_{\mathrm{L}^2(\Omega)} \\
			= & \left \langle \partial_{\nu}\left( \frac{V_{\mu_1,\mu_2,k,q_i}}{d^{a-1}}\right), \frac{U}{d^{a-1}}\right\rangle_{\mathrm{L}^2(\partial\Omega)}-\left \langle  \frac{V_{\mu_1,\mu_2,k,q_i}}{d^{a-1}}, \partial_{\nu}\left(\frac{U}{d^{a-1}}\right)\right\rangle_{\mathrm{L}^2(\partial\Omega)}.
		\end{align*}
		For $i=1,2$, the above equality implies
		\begin{align*}
			&\left\langle \big((-\Delta)^a+q_i\big) V_{\mu_1,\mu_2,k,q_i}, U\right\rangle_{\mathrm{L}^2(\Omega)} - \left\langle \big((-\Delta)^a+q_i\big) U, V_{\mu_1,\mu_2,k,q_i}\right\rangle_{\mathrm{L}^2(\Omega)} \\
			= & \left \langle \partial_{\nu}\left( \frac{V_{\mu_1,\mu_2,k,q_i}}{d^{a-1}}\right), \frac{U}{d^{a-1}}\right\rangle_{\mathrm{L}^2(\partial\Omega)}-\left \langle  \frac{V_{\mu_1,\mu_2,k,q_i}}{d^{a-1}}, \partial_{\nu}\left(\frac{U}{d^{a-1}}\right)\right\rangle_{\mathrm{L}^2(\partial\Omega)}.
		\end{align*}
		Thanks to the $\mathrm{L}^2$ convergence \eqref{four convergence}, by passing the limit $k\to+\infty$, we obtain
		\begin{align*}
			&\left\langle \big((-\Delta)^a+q_i\big) (v_{\mu_1,i}-v_{\mu_2,i}), U\right\rangle_{\mathrm{L}^2(\Omega)} - \left\langle \big((-\Delta)^a+q_i\big) U, v_{\mu_1,i}-v_{\mu_2,i}\right\rangle_{\mathrm{L}^2(\Omega)} \\
			&=  \Bigg \langle \Gamma(a)\Gamma(a+1)\sum\limits_{n\in\mathbb{N}} \frac{(\mu_1-\mu_2)}{(\mu_1-\lambda_n)(\mu_2-\lambda_n)},\left\langle F,\frac{\phi_{n,q_i}}{d^{a-1}}\right\rangle \frac{\phi_{k,q_i}}{d^a}, \frac{U}{d^{a-1}}\Bigg\rangle_{\mathrm{L}^2(\partial\Omega)}\\
			&-\left \langle  \frac{v_{\mu_1,i}-v_{\mu_2,i}}{d^{a-1}}, \partial_{\nu}\left(\frac{U}{d^{a-1}}\right)\right\rangle_{\mathrm{L}^2(\partial\Omega)}.
		\end{align*}
		Equation \eqref{auxiliary equation U}, corresponding to $U$, yields
		\begin{align*}
			\left\langle \big((-\Delta)^a+q_i\big) v_{\mu_1,i}-v_{\mu_2,i}, U\right\rangle_{\mathrm{L}^2(\Omega)} -& \left\langle \big((-\Delta)^a+q_i\big) U, v_{\mu_1,i}-v_{\mu_2,i}\right\rangle_{\mathrm{L}^2(\Omega)} \\
			=  \Bigg \langle \Gamma(a)\Gamma(a+1)\sum\limits_{n\in\mathbb{N}} \frac{(\mu_1-\mu_2)}{(\mu_1-\lambda_n)(\mu_2-\lambda_n)},&\left\langle F,\frac{\phi_{n,q_i}}{d^{a-1}}\right\rangle \frac{\phi_{k,q_i}}{d^a}, f\Bigg\rangle_{\mathrm{L}^2(\partial\Omega)}\\
			&-\left \langle  \frac{v_{\mu_1,i}-v_{\mu_2,i}}{d^{a-1}}, \partial_{\nu}\left(\frac{U}{d^{a-1}}\right)\right\rangle_{\mathrm{L}^2(\partial\Omega)}.
		\end{align*}
		Since $f\in C_c^{\infty}(\Omega)$ is arbitrary, we have the following: for $i=1,2$
		\[
		\partial_{\nu}\left( \frac{v_{\mu_1,i}-v_{\mu_2,i}}{d^{a-1}}\right)= \Gamma(a)\Gamma(a+1)\sum\limits_{n\in\mathbb{N}} \frac{(\mu_1-\mu_2)}{(\mu_1-\lambda_n)(\mu_2-\lambda_n)},\left\langle F,\frac{\phi_{n,q_i}}{d^{a-1}}\right\rangle \frac{\phi_{k,q_i}}{d^a}.
		\]
		Due to the equality of Dirichlet to Neumann boundary data \eqref{Borg Levinson fractional criteria}, we have that
		\[
		\partial_{\nu}\left( \frac{v_{\mu_1,1}-v_{\mu_2,1}}{d^{a-1}}\right)= \partial_{\nu}\left( \frac{v_{\mu_1,2}-v_{\mu_2,2}}{d^{a-1}}\right).
		\]
	\end{proof}
	Our next proposition establishes equality of two \emph{Dirichlet to Neumann (DN) maps} corresponding to two different potentials.
	\begin{Prop}\label{equality of elliptic DN map}
		Let $\mu_1\in \mathbb{C}\setminus \{\lambda_n\}|_{n\in\mathbb{N}}$ and let $a\in\left(\frac12,1\right)$. For $i=1,2$, let $v_{\mu_1,i}$ denotes the solution to \eqref{equation in complex variable outside spectrum} corresponding to $\lambda=\mu_1$. Then, we have the following result: 
		\[
		\partial_{\nu}\left( \frac{v_{\mu_1,1}}{d^{a-1}}\right)= \partial_{\nu}\left( \frac{v_{\mu_1,2}}{d^{a-1}}\right).
		\]
	\end{Prop}
	\begin{proof}
		Proposition \eqref{equality of DN map for difference of solution} ensures that 
		\[
		\partial_{\nu}\left( \frac{v_{\mu_1,1}-v_{\mu_2,1}}{d^{a-1}}\right)= \partial_{\nu}\left( \frac{v_{\mu_1,2}-v_{\mu_2,2}}{d^{a-1}}\right), \quad \forall \, \mu_1,\mu_2\in\mathbb{C}\setminus\{\lambda_n\}_{n\in\mathbb{N}}.
		\]
		The function $v_{\mu_2,1}-v_{\mu_2,2}$, satisfies the following equation:
		\begin{equation}\label{auxiliary equation on V}
			\left \{
			\begin{aligned}
				(-\Delta)^a(v_{\mu_2,1}-v_{\mu_2,2})-\mu_2 (v_{\mu_2,1}-v_{\mu_2,2})=& q_2v_2-q_1v_1 \qquad \mbox{in} \ \Omega\\
				v_{\mu_2,1}-v_{\mu_2,2}=&0 \qquad \qquad \qquad \ \mbox{in} \ \mathbb{R}^N\setminus\overline{\Omega}\\
				\frac{v_{\mu_2,1}-v_{\mu_2,2} }{d^{a-1}}=& 0 \qquad \qquad \qquad \  \ \mbox{on}\ \partial\Omega.
			\end{aligned}
			\right .
		\end{equation}
		Due to the discreteness of eigen spectrum, the spectral problem corresponding to the operator $((-\Delta)^a-\mu_2)$, can be expressed in the following way
		\begin{equation}\label{spectral elliptic}
			\left \{
			\begin{aligned}
				(-\Delta)^a\phi_{\mu,n}-\mu_2 \phi_{\mu,n}=& \lambda_{\mu,n}\phi_{\mu,n} \qquad \mbox{in} \ \Omega\\
				\phi_{\mu,n}=&0 \qquad \mbox{in} \ \mathbb{R}^N\setminus\Omega,
			\end{aligned}
			\right .
		\end{equation}	
		where $\phi_{\mu,n}$ is the normalized $($in $\mathrm{L}^2$ sense$)$ eigenfunction corresponding to the eigenvalue $\lambda_{\mu,n}$. Let $\{\lambda_{0,n}\}|_{n\in\mathbb{N}}$ be the eigen spectrum of the same spectral problem \eqref{spectral elliptic} when $\mu_2=0$ and $\phi_{0,n}$ is the eigenfunction corresponding to $\lambda_{0,n}$. Let $\mu_2 \not\in\{\lambda_{0,n}\}|_{n\in\mathbb{N}}$, then
		\[
		\lambda_{\mu,n}=\lambda_{0,n}-\mu_2, \ \ \ \mbox{and} \ \ \phi_{\mu,n}=\phi_{0,n}\in 	H^{a(2a)}(\overline{\Omega}).
		\]
		Moreover, $\{\phi_{0,n}\}|_{n\in\mathbb{N}}$ forms an orthonormal basis in $\mathrm{L}^2(\Omega)$. Furthermore, as $v_{\mu_2,1}-v_{\mu_2,2}\in\mathrm{L}^2(\Omega)$, it admits the following representation:
		\[
		v_{\mu_2,1}-v_{\mu_2,2}= \sum\limits_{n\in\mathbb{N}} \langle v_{\mu_2,1}-v_{\mu_2,2},\phi_{0,n}\rangle_{\mathrm{L}^2(\Omega)}\phi_{0,n}.
		\]
		Next we use integration by parts formula as in \eqref{integration by parts dirichlt}. We have that
		\begin{align*}
			-\Gamma(a)\Gamma(a+1)&\int_{\partial\Omega} \frac{v_{\mu_2,1}-v_{\mu_2,2}}{d^{a-1}} \frac{\phi_{0,n}}{d^{a}}\, \rm{d}\sigma\\
			=& \int_{\Omega} \phi_{0,n}(-\Delta)^a (v_{\mu_2,1}-v_{\mu_2,2}) \, \rm{d}x- \int_{\Omega}(v_{\mu_2,1}-v_{\mu_2,2})(-\Delta)^a \phi_{0,n} \, \rm{d}x\\
			= &\int_{\Omega} \phi_{0,n}\left((-\Delta)^a-\mu_2\right) (v_{\mu_2,1}-v_{\mu_2,2}) \, \rm{d}x\\
			-& \int_{\Omega}(v_{\mu_2,1}-v_{\mu_2,2})\left((-\Delta)^a-\mu_2\right) \phi_{0,n} \, \rm{d}x\\
			= & \int_{\Omega} (q_2v_{\mu_2,2}-q_1v_{\mu_2,1})\phi_{0,n} \, \rm{d}x-(\lambda_{n}-\mu_2) \int_{\Omega} (v_{\mu_2,1}-v_{\mu_2,2})\phi_{n,q_i} \, \rm{d}x.
		\end{align*}
		Hence $(v_{\mu_2,1}-v_{\mu_2,2})$ can be expressed as 
		\[
		v_{\mu_2,1}-v_{\mu_2,2}=  \sum\limits_{n\in\mathbb{N}}\frac{1}{\lambda_{0,n}-\mu_2}\left\langle (q_2v_{\mu_2,2}-q_1v_{\mu_2,1}),\phi_{0,n}\right\rangle_{\mathrm{L}^2(\Omega)} \phi_{0,n},
		\]
		with 
		\[
		\sum_{n\in\mathbb{N}} \Big| \left\langle (q_2v_{\mu_2,2}-q_1v_{\mu_2,1}),\phi_{0,n}\right\rangle\Big|^2 <+\infty,
		\]
		where the last finiteness follows from the fact that $ (q_2v_{\mu_2,2}-q_1v_{\mu_2,1})\in \mathrm{L}^2(\Omega)$ [see \cite{grubb2014local}]. 
		Let us consider the following function, obtained as the projection onto the first $k$-dimensional subspace:
		\[
		V=\sum\limits_{n=1}^k\frac{1}{\lambda_{0,n}-\mu_2}\left\langle (q_2v_{\mu_2,2}-q_1v_{\mu_2,1}),\phi_{0,n}\right\rangle_{\mathrm{L}^2(\Omega)} \phi_{0,n}.
		\]
		We further have 
		\[
		\frac{v_{\mu_2,1}-v_{\mu_2,2}}{d^{a-1}}\Big|_{\partial\Omega}\equiv \frac{V}{d^{a-1}}\Big|_{\partial\Omega}\equiv 0.
		\]
		Hence, $\displaystyle{\frac{V}{d^{a-1}}\Big|_{\partial\Omega}\xrightarrow{k\to+\infty}\frac{v_{\mu_2,1}-v_{\mu_2,2}}{d^{a-1}}\Big|_{\partial\Omega}}$ in $\mathrm{L}^2(\Omega)$. Moreover, $V$ satisfies the following equation:
		\begin{equation*}
			\left \{
			\begin{aligned}
				(-\Delta)^a V+\mu_2V =& \sum\limits_{n=1}^k\frac{\lambda_{0,n}}{\lambda_{0,n}-\mu_2}\left\langle (q_2v_{\mu_2,2}-q_1v_{\mu_2,1}),\phi_{0,n}\right\rangle_{\mathrm{L}^2(\Omega)} \phi_{0,n} \ \ \mbox{in} \ \Omega\\
				V=& 0 \qquad \qquad \qquad \qquad \qquad \qquad \qquad \qquad \qquad \qquad \quad  \mbox{in} \ \mathbb{R}^N\setminus\Omega.
			\end{aligned}
			\right .
		\end{equation*} 
		Thanks to the fact that $\displaystyle{\sup\limits_{n\in\mathbb{N}}\left|\frac{\lambda_{0,n}}{\lambda_{0,n}-\mu_2}\right|<+\infty}$	and
		\begin{align*}
			&\left\Vert \sum\limits_{n\in\mathbb{N}}\frac{1}{\lambda_{0,n}-\mu_2}\left\langle (q_2v_{\mu_2,2}-q_1v_{\mu_2,1}),\phi_{0,n}\right\rangle_{\mathrm{L}^2(\Omega)} \frac{\phi_{0,n}}{d^a}\right\Vert_{\mathrm{L}^2(\partial\Omega)}^2\\
			& \leq \sum\limits_{n\in\mathbb{N}}\left|\frac{1}{\lambda_{0,n}-\mu_2}\right|^2 \left| \left\langle (q_2v_{\mu_2,2}-q_1v_{\mu_2,1}),\phi_{0,n}\right\rangle_{\mathrm{L}^2(\Omega)} \right|^2 \left\Vert \frac{\phi_{0,n}}{d^a} \right\Vert_{\mathrm{L}^2(\partial\Omega)}^2\\
			& \leq C(\Omega) \sum\limits_{n\in\mathbb{N}}\left|\frac{\lambda_{0,n}}{\lambda_{0,n}-\mu_2}\right|^2 \left| \left\langle (q_2v_{\mu_2,2}-q_1v_{\mu_2,1}),\phi_{0,n}\right\rangle_{\mathrm{L}^2(\Omega)} \right|^2 \big\Vert \phi_{0,n} \big\Vert_{\mathrm{L}^2(\Omega)}^2\\
			& \leq C(\Omega)\sup\limits_{n\in\mathbb{N}}\left\{\left|\frac{\lambda_{0,n}}{\lambda_{0,n}-\mu_2}\right|^2\right\}\sum\limits_{n\in\mathbb{N}} \left| \left\langle (q_2v_{\mu_2,2}-q_1v_{\mu_2,1}),\phi_{0,n}\right\rangle_{\mathrm{L}^2(\Omega)} \right|^2\\
			& <+\infty \quad (C(\Omega)>0 \ \mbox{as in}\ \eqref{eigen value regularity}).
		\end{align*}
		We can proceed exactly like Proposition \ref{equality of DN map for difference of solution} to obtain the Dirichlet to Neumann boundary data of $v_{\mu_2,1}-v_{\mu_2,2}$. We obtain
		\begin{align}\label{DN to 0 sum}
			\partial_{\nu}\left(\frac{v_{\mu_2,1}-v_{\mu_2,2}}{d^{a-1}}\right)= \sum\limits_{n\in\mathbb{N}}\frac{1}{\lambda_{0,n}-\mu_2}\left\langle (q_2v_{\mu_2,2}-q_1v_{\mu_2,1}),\phi_{0,n}\right\rangle_{\mathrm{L}^2(\Omega)} \frac{\phi_{0,n}}{d^a}.
		\end{align}
		Here equality is in the sense of $\mathrm{L}^2(\Omega)$. Since $\mu_2\in\mathbb{C}\setminus\Big(\{\lambda_n\}\cup \{\lambda_{0,n}\}\Big)$ arbitrary, we can choose
		\[
		\mu_{2,l}= -l+\min\{\lambda_{0,1},\lambda_{1}\}+i, \quad \mbox{where}\ l\in\mathbb{N}, \, i^2=-1. 
		\]
		Thanks to the fact that $\{\lambda_n\}|_{n\in\mathbb{N}}\cup\{\lambda_{0,n}\}|_{n\in\mathbb{N}}\subset\mathbb{R}$ and that the monotonicity character of eigen spectrum
		\begin{align*}
			&\lambda_1<\lambda_2\leq \cdots\leq \lambda_n\leq\cdots +\infty\\
			&\lambda_{0,1}<\lambda_{0,2}\leq \cdots\leq \lambda_{0,n}\leq\cdots +\infty,
		\end{align*}
		the above choice of $\mu_{2,n}$ is justified. The following fact holds
		\[
		dist_{\mathbb{C}} \Big( \mu_{2,l}, \{\lambda_n\}|_{n\in\mathbb{N}}\cup\{\lambda_{0,n}\}|_{n\in\mathbb{N}} \Big) \leq dist_{\mathbb{C}} \Big(\min\{\lambda_{0,1}, \lambda_1\}+i, \min\{\lambda_{0,1}, \lambda_1\}+i\cdot0 \Big).
		\]
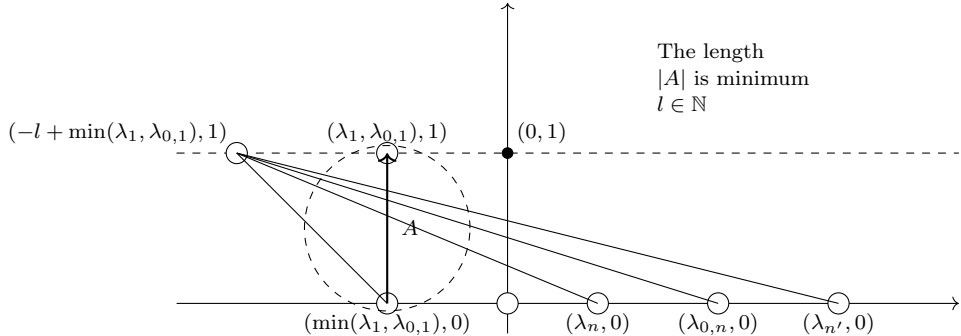
		\begin{tikzpicture}[scale=2, every node/.style={font=\footnotesize}]
			\draw[->] (-2.2,0) -- (3,0) node[right] {};
			\draw[->] (0,-0.2) -- (0,2) node[above] {};
			
			\draw[dashed] (-2.2,1) -- (3,1);
			\coordinate (P0) at (0,0);   %
			\coordinate (P1) at (-1.8,1);   
			\coordinate (P2) at (-0.8,1);   
			\coordinate (P3) at (-0.8,0);   
			\coordinate (P4) at (0,1);      
			\coordinate (P5) at (0.6,0);    
			\coordinate (P6) at (1.4,0);    
			\coordinate (P7) at (2.2,0);    
			
			\foreach \pt in {P0,P1,P2,P3,P5,P6,P7}
			\draw[fill=white] (\pt) circle (2pt);
			\filldraw (P4) circle (1pt);  
			
			\draw[->, thick] (P3) -- (P2);
			\node at (-0.65,0.5) {$A$};
			
			\draw[dashed] (-0.8,0.5) circle [radius=0.55];
			
			\foreach \pt in {P3,P5,P6,P7}
			\draw (P1) -- (\pt);
			
			\node[above left] at (P1) {$(-l + \min(\lambda_1, \lambda_{0,1}), 1)$};
			\node[above] at (P2) {$(\lambda_1, \lambda_{0,1}), 1)$};
			\node[below] at (P3) {$(\min(\lambda_1, \lambda_{0,1}), 0)$};
			\node[above right] at (P4) {$(0, 1)$};
			\node[below] at (P5) {$(\lambda_n, 0)$};
			\node[below] at (P6) {$(\lambda_{0,n}, 0)$};
			\node[below] at (P7) {$(\lambda_{n'}, 0)$};
			
			\node[align=left] at (1.5,1.5) {The length\\$|A|$ is minimum\\ $l\in\mathbb{N}$};
		\end{tikzpicture}
		
		\vspace{.2cm}
		
		Let us denote $\displaystyle{c^*:= \min\{\lambda_{0,1}, \lambda_1\}+i \not\in \{\lambda_n\}|_{n\in\mathbb{N}}\cup\{\lambda_{0,n}\}|_{n\in\mathbb{N}}}$. We have
		\begin{align*}
			&\left\Vert \partial_{\nu}\left(\frac{v_{\mu_{2,l},1}-v_{\mu_{2,l},2}}{d^{a-1}}\right)\right\Vert_{\mathrm{L}^2(\partial\Omega)}^2\\
			&=  \left\Vert\sum\limits_{n\in\mathbb{N}}\frac{1}{\lambda_{0,n}-\mu_{2,n}}\left\langle (q_2v_{\mu_{2,l},2}-q_1v_{\mu_{2,l},1}),\phi_{0,n}\right\rangle_{\mathrm{L}^2(\Omega)} \frac{\phi_{0,n}}{d^a}\right\Vert_{\mathrm{L}^2(\partial\Omega)}^2\\
			& \leq \sum\limits_{n\in\mathbb{N}}\left|\frac{1}{\lambda_{0,n}-\mu_{2,n}}\right|^2 \left| \left\langle (q_2v_{\mu_{2,l},2}-q_1v_{\mu_{2,l},1}),\phi_{0,n}\right\rangle_{\mathrm{L}^2(\Omega)} \right|^2 \left\Vert \frac{\phi_{0,n}}{d^a} \right\Vert_{\mathrm{L}^2(\partial\Omega)}^2\\
			& \leq \sum\limits_{n\in\mathbb{N}}\left|\frac{\lambda_{0,n}}{\lambda_{0,n}-\mu_{2,n}}\right|^2 \left| \left\langle (q_2v_{\mu_{2,l},2}-q_1v_{\mu_{2,l},1}),\phi_{0,n}\right\rangle_{\mathrm{L}^2(\Omega)} \right|^2 \big\Vert \phi_{0,n} \big\Vert_{\mathrm{L}^2(\Omega)}^2\\
			& \leq \sup\limits_{n\in\mathbb{N}}\left\{\left|\frac{\lambda_{0,n}}{\lambda_{0,n}-\mu_{2,n}}\right|^2\right\}\sum\limits_{n\in\mathbb{N}} \left| \left\langle (q_2v_{\mu_{2,l},2}-q_1v_{\mu_{2,l},1}),\phi_{0,n}\right\rangle_{\mathrm{L}^2(\Omega)} \right|^2\\
			&\leq \sup\limits_{n\in\mathbb{N}}\left\{\left|\frac{\lambda_{0,n}}{\lambda_{0,n}-c^*}\right|^2\right\}\sum\limits_{n\in\mathbb{N}} \left| \left\langle (q_2v_{\mu_{2,l},2}-q_1v_{\mu_{2,l},1}),\phi_{0,n}\right\rangle_{\mathrm{L}^2(\Omega)} \right|^2.	
		\end{align*} 	
		We analyze the last sum
		\begin{align*}
			\Vert q_2v_{\mu_{2,l},2}-q_1 v_{\mu_{2,l},1} \Vert_{\mathrm{L}^2(\Omega)}^2=  &\sum\limits_{n\in\mathbb{N}} \left| \left\langle (q_2v_{\mu_{2,l},2}-q_1v_{\mu_{2,l},1}),\phi_{0,n}\right\rangle_{\mathrm{L}^2(\Omega)} \right|^2\\
			\leq & 2\sup\limits_{x\in\Omega}\{|q_1|+|q_2|\}\Big( \Vert v_{\mu_{2,l},1}\Vert_{\mathrm{L}^2(\Omega)}^2+ \Vert v_{\mu_{2,l},1}\Vert_{\mathrm{L}^2(\Omega)}^2\Big).
		\end{align*}
		Representation of solution \eqref{representation of elliptic solution} yields
		\begin{align*}
			\Vert v_{\mu_{2,l},i}\Vert_{\mathrm{L}^2(\Omega)}^2 \leq&  \sum_{n\in\mathbb{N}}  \frac{1}{| \mu_{2,l}-\lambda_n|^2}\left\langle F,\frac{\phi_{n,q_i}}{d^a}\right\rangle_{\mathrm{L}^2(\partial\Omega)}^2, \quad \mbox{for} \ i=1,2\\
			\leq&  \sum_{n\in\mathbb{N}}  \frac{1}{| c^*-\lambda_n|^2}\left\langle F,\frac{\phi_{n,q_i}}{d^a}\right\rangle_{\mathrm{L}^2(\partial\Omega)}^2, \quad \mbox{for} \ i=1,2.
		\end{align*}
		It yields
		\begin{align*}
			&\left\Vert \partial_{\nu}\left(\frac{v_{\mu_{2,l},1}-v_{\mu_{2,l},2}}{d^{a-1}}\right)\right\Vert_{\mathrm{L}^2(\partial\Omega)}^2 \\
			&\leq \sup\limits_{n\in\mathbb{N}}\left\{\left|\frac{\lambda_{0,n}}{\lambda_{0,n}-c^*}\right|^2\right\} \sum_{i=1,2} \sum_{n\in\mathbb{N}}  \frac{1}{| c^*-\lambda_n|^2}\left\langle F,\frac{\phi_{n,q_i}}{d^a}\right\rangle_{\mathrm{L}^2(\partial\Omega)}^2 <+\infty,
		\end{align*}
		where the finiteness follows from \eqref{L2 norm elliptic solution} as $c^*\not\in \{\lambda_n\}|_{n\in\mathbb{N}}$. Hence, we can pass the limit $l\to+\infty$ inside the sum \eqref{DN to 0 sum} in $\mathrm{L}^2$ sense. Hence
		\[
		\lim\limits_{l\to +\infty} \left(\frac{v_{\mu_{2,l},1}-v_{\mu_{2,l},2}}{d^{a-1}}\right)= \sum\limits_{n\in\mathbb{N}}\lim\limits_{l\to +\infty}\frac{1}{\lambda_{0,n}-\mu_{2,l}}\left\langle (q_2v_{\mu_2,2}-q_1v_{\mu_2,1}),\phi_{0,n}\right\rangle_{\mathrm{L}^2(\Omega)} \frac{\phi_{0,n}}{d^a}=0.
		\] 
		Proposition \eqref{equality of DN map for difference of solution} ensures that 
		\[
		\partial_{\nu}\left( \frac{v_{\mu_1,1}-v_{\mu_{2,l},1}}{d^{a-1}}\right)= \partial_{\nu}\left( \frac{v_{\mu_1,2}-v_{\mu_{2,l},2}}{d^{a-1}}\right), \quad \forall \, \mu_1,\mu_{2,l}\in\mathbb{C}\setminus\{\lambda_n\}_{n\in\mathbb{N}}.
		\]	
		We choose 
		\[
		\mu_{2,l}= -l+\min\{\lambda_{0,1},\lambda_{1}\}+i, \quad \mbox{where}\ l\in\mathbb{N}, \, i^2=-1. 
		\]
		We have 
		\[
		\partial_{\nu}\left( \frac{v_{\mu_1,1}-v_{\mu_{l},2}}{d^{a-1}}\right)= \partial_{\nu}\left( \frac{v_{\mu_{2,l},1}-v_{\mu_{2,l},2}}{d^{a-1}}\right), \qquad \mbox{for all} \ l\in\mathbb{N}.
		\]
		Taking the limit as $l\to+\infty$, we obtain
		\[
		\partial_{\nu}\left( \frac{v_{\mu_1,1}-v_{\mu_{l},2}}{d^{a-1}}\right)= 0.
		\]
		Hence
		\[
		\partial_{\nu}\left( \frac{v_{\mu_1,1}}{d^{a-1}}\right)= \partial_{\nu}\left( \frac{v_{\mu_1,2}}{d^{a-1}}\right).
		\]
	\end{proof}
	In the next section, we demonstrate how the equality of \emph{Dirichlet to Neumann (DN) maps} for the fractional elliptic case leads to the the equality of \emph{Dirichlet to Neumann (DN) maps} for the non-local parabolic case.
	
	\section{Borg-Levinson problem to parabolic inverse problem}\label{Borg-Levinson problem to parabolic inverse problem}
	
	Let $w_i$ satisfies the following equation: for $i=1,2$
	\begin{equation}\label{varying boundary data}
		\left\{
		\begin{aligned}
			\partial_t w_{i}+(-\Delta)^a w_i+ q_iw_i=&0 \qquad \qquad  \mbox{in}\ (0,T)\times\Omega\\
			w_i=&0 \qquad \qquad  \mbox{in} \ (0,T)\times\mathbb{R}^N\setminus\overline{\Omega}\\
			\lim\limits_{\substack{x\to y\\ y\in\partial\Omega}}\frac{w_i}{d^{a-1}}(y)=& f \qquad \qquad \mbox{on} \ (0,T)\times\partial\Omega\\
			w_i(0,\cdot)=&0 \qquad \qquad  \mbox{in} \ \Omega.
		\end{aligned}
		\right .
	\end{equation}
	Thanks to Theorem \ref{s-transmission regularity, boundary data} \cite{grubb2023resolvents} and the continuous inclusion of $H^{(a-1)(2a)}(\overline{\Omega}) \subset H^{2a}(\Omega)$ (Lemma \ref{Further chracterization of mu transmission space}), we have that: for $i=1,2$  
	\begin{equation*}
		\left \{
		\begin{aligned}
			w_i \in & \mathrm{L}^2((0,T); \mathrm{L}^2(\Omega))\\
			\partial_t w_i \in & \mathrm{L}^2((0,T); \mathrm{L}^2(\Omega))\\
			(-\Delta)^a w_i \in & \mathrm{L}^2((0,T); \mathrm{L}^2(\Omega)).
		\end{aligned}
		\right .
	\end{equation*}
	For $i=1,2$, let $\{\lambda_{n,q_i}\}|_{n\in\mathbb{N}}$ be the set of eigenvalues (counted with multiplicity) and $\{\phi_{n,q_i}\}|_{n\in\mathbb{N}}$ be the set of normalized eigenfunctions for the operator $(-\Delta)^a+q_i$, satisfying \eqref{1st elliptic eigenvalue problem}. The two operators $(-\Delta)^a+q_1$ and $(-\Delta)^a+q_2$ satisfy Borg-Levinson criteria \ref{Borg Levinson fractional criteria} i.e., they have the same discrete spectrum of eigenvalues and the same Dirichlet boundary datum. We have $\{\lambda_{n,q_i}\}|_{n\in\mathbb{N}}=\{\lambda_{n}\}|_{n\in\mathbb{N}}$ with the following monotonicity property
	\[
	\lambda_1<\lambda_2\leq \cdots\leq \lambda_n\leq \cdots +\infty.
	\]
	Moreover, $\{\phi_{n,q_i}\}$ forms a basis in $\mathrm{L}^2(\Omega)$, for $i=1,2$. Hence for each $t>0$, $w_i$ can be expressed as
	\[
	w_i(t)= \sum\limits_{n\in\mathbb{N}} \langle w_i(t), \phi_{n,q_i}\rangle_{\mathrm{L}^2(\Omega)}\phi_{n,q_i} \quad \mbox{for} \ i=1,2.
	\]
	Using $\phi_{n,q_i}$ as a test function in the equation of $w_i$, we obtain
	\[
	\langle \partial_t w_i(t), \phi_{n,q}\rangle_{\mathrm{L}^2(\Omega)}+ \lambda_n \langle w_i, \phi_n\rangle_{\mathrm{L}^2(\Omega)}= -\Gamma(a)\Gamma(a+1) \left\langle f, \frac{\phi_{n,q_i}}{d^a}\right\rangle_{\mathrm{L}^2(\partial\Omega)}, \quad \mbox{for} \ i=1,2. 
	\]
	Let $\chi\in C_c^{\infty}(0,T)$, then 
	\begin{align*}
		\int_0^T \left(\int_{\Omega} \partial_t w_i \phi_{n,q_i} \right) \chi \, \rm{d}t\stackrel{\text{Fubini}}{=}& \int_{\Omega} \left(\int_0^T \partial_t w \chi \, \rm{d}t\right) \phi_{n,q_i}\\
		= -\int_{\Omega} \left(\int_0^T w \partial_t\chi \, \rm{d}t\right)&\phi_{n,q_i} \stackrel{\text{Fubini}}{=} -\int_0^T \left( \int_{\Omega} w\phi_{n,q_i}\right) \partial_t \chi \, \rm{d}t.
	\end{align*}
	Hence the weak derivative of $\displaystyle{\int_{\Omega} w_i\phi_{n,q_i}}$ is $\displaystyle{-\int_{\Omega} \partial_t w_i\phi_{n,q_i}}$. Hence, for $i=1,2$, $w_i$ satisfies the following equation:
	\begin{equation*}
		\left\{
		\begin{aligned}
			\partial_t\langle  w_i(t), \phi_{n,q}\rangle_{\mathrm{L}^2(\Omega)}+ \lambda_n \langle w_i, \phi_n\rangle_{\mathrm{L}^2(\Omega)}= &-\Gamma(a)\Gamma(a+1) \left\langle f, \frac{\phi_{n,q_i}}{d^a}\right\rangle_{\mathrm{L}^2(\partial\Omega)} \quad \mbox{in} \ t\in[0,T]\\
			\langle w_i, \phi_{n,q_i}\rangle_{\mathrm{L}^2(\Omega)} (t=0)=& 0 \qquad \qquad \qquad \qquad \qquad \qquad \qquad \ \ \  \mbox{at} \ t=\{0\}.
		\end{aligned}
		\right. 
	\end{equation*}
	Here, the source term on the right hand side is smooth with respect to `$t$'. Hence the above equation has a classical solution, given by 
	\[
	\langle w_i, \phi_{n,q_i}\rangle_{\mathrm{L}^2(\Omega)}(t)= -\Gamma(a)\Gamma(a+1)\int_{0}^{t} e^{-\lambda_n(t-s)} \left\langle f, \frac{\phi_{n,q_i}}{d^a}\right\rangle_{\mathrm{L}^2(\partial\Omega)}(s) \, \rm{d}s, \ \ \forall \, t\in[0,T). 
	\]
	From the equation, we also have that
	\begin{align*}
		\langle \partial_t w_i, \phi_{n,q_i}&\rangle_{\mathrm{L}^2(\Omega)}(t)=-\Gamma(a)\Gamma(a+1) \left\langle f, \frac{\phi_{n,q_i}}{d^a}\right\rangle_{\mathrm{L}^2(\partial\Omega)}(t)\\
		& +\Gamma(a)\Gamma(a+1)\int_{0}^{t} \lambda_ne^{-\lambda_n(t-s)} \left\langle f, \frac{\phi_{n,q_i}}{d^a}\right\rangle_{\mathrm{L}^2(\partial\Omega)}(s) \, \rm{d}s, \ \ \ \forall \, t\in(0,T).
	\end{align*}
	Let $0<\tilde{T}<T$, be such that $supp(f)\in C_c^{\infty}((0,\tilde{T}]\times\Omega)$. As $\partial_t w_i\in \mathrm{L}^2((0,T);\mathrm{L}^2(\Omega))$, there exists $T^*_i\in[\tilde{T},T]$ such that
	\begin{align}\label{finiteness of space integral on a particular time}
		\Big\Vert \partial_t w_i \Big\Vert_{\mathrm{L}^2(\Omega)}(T^*_{i}) <+\infty, \quad \mbox{for} \ i=1,2. 
	\end{align}
	We, extend the solution $w_i(t)$ and $\partial_t(w)$ to $(0,+\infty)$,  i.e.,
	\begin{equation}\label{extended solution to infty}
		\left \{
		\begin{aligned}
			\langle w_i, \phi_{n,q_i}\rangle_{\mathrm{L}^2(\Omega)}(t)= & -\Gamma(a)\Gamma(a+1)\int_{0}^{t} e^{-\lambda_n(t-s)} \left\langle f, \frac{\phi_{n,q_i}}{d^a}\right\rangle_{\mathrm{L}^2(\partial\Omega)}(s) \, \rm{d}s, \ \ \forall \, t\geq 0,\\
			\langle \partial_t w_i, \phi_{n,q_i}\rangle_{\mathrm{L}^2(\Omega)}(t)=&-\Gamma(a)\Gamma(a+1) \left\langle f, \frac{\phi_{n,q_i}}{d^a}\right\rangle_{\mathrm{L}^2(\partial\Omega)}(t)\\
			& +\Gamma(a)\Gamma(a+1)\int_{0}^{t} \lambda_ne^{-\lambda_n(t-s)} \left\langle f, \frac{\phi_{n,q_i}}{d^a}\right\rangle_{\mathrm{L}^2(\partial\Omega)}(s) \, \rm{d}s,
		\end{aligned}
		\right .
	\end{equation}
	where the second one holds for $t\in(0,+\infty)$. Let $t\geq T$, then 
	\begin{align*}
		\langle \partial_t w_i, \phi_{n,q_i}\rangle_{\mathrm{L}^2(\Omega)}(t)= & -\Gamma(a)\Gamma(a+1) \left\langle f, \frac{\phi_{n,q_i}}{d^a}\right\rangle_{\mathrm{L}^2(\partial\Omega)}(t)\\
		& +\Gamma(a)\Gamma(a+1)\int_{0}^{t} \lambda_ne^{-\lambda_n(t-s)} \left\langle f, \frac{\phi_{n,q_i}}{d^a}\right\rangle_{\mathrm{L}^2(\partial\Omega)}(s) \, \rm{d}s\\
		=&\Gamma(a)\Gamma(a+1)\int_{0}^{t} \lambda_ne^{-\lambda_n(t-s)} \left\langle f, \frac{\phi_{n,q_i}}{d^a}\right\rangle_{\mathrm{L}^2(\partial\Omega)}(s) \, \rm{d}s\\
		=&\Gamma(a)\Gamma(a+1)\int_{0}^{T^*_i} \lambda_ne^{-\lambda_n(t-s)} \left\langle f, \frac{\phi_{n,q_i}}{d^a}\right\rangle_{\mathrm{L}^2(\partial\Omega)}(s) \, \rm{d}s\\
		= & e^{-\lambda_n(t-T^*_i)} \int_{0}^{T^*_i} \lambda_ne^{-\lambda_n(T^*_i-s)} \left\langle f, \frac{\phi_{n,q_i}}{d^a}\right\rangle_{\mathrm{L}^2(\partial\Omega)}(s) \, \rm{d}s\\
		= & e^{-\lambda_n(t-T^*_i)} \langle \partial_t w_i, \phi_{n,q_i}\rangle_{\mathrm{L}^2(\Omega)}(T^*_i), \quad \mbox{for} \ i=1,2.
	\end{align*}
	Now
	\begin{align*}
		\int_0^{+\infty} \int_{\Omega} | \partial_t w_i |^2& \, \rm{d}x \, \rm{d}t =  \int_0^{T} \int_{\Omega} | \partial_t w_i |^2 \, \rm{d}x \, \rm{d}t + \int_T^{+\infty} \int_{\Omega} | \partial_t w_i |^2 \, \rm{d}x \, \rm{d}t\\
		= & \int_0^{T} \int_{\Omega} | \partial_t w_i |^2 \, \rm{d}x \, \rm{d}t + \int_{T}^{+\infty} \sum_{n\in\mathbb{N}} e^{-2\lambda_n(t-T^*_i)} \langle \partial_t w_i, \phi_{n,q_i}\rangle^2_{\mathrm{L}^2(\Omega)}(T^*_i)\\
		\stackrel{\text{Monotone Convergence}}{=} & \int_0^{T} \int_{\Omega} | \partial_t w_i |^2 \, \rm{d}x \, \rm{d}t + \sum_{n\in\mathbb{N}}\int_{T}^{+\infty}  e^{-2\lambda_n(t-T^*_i)} \langle \partial_t w_i, \phi_{n,q_i}\rangle^2_{\mathrm{L}^2(\Omega)}(T^*_i)\\
		=& \int_0^{T} \int_{\Omega} | \partial_t w_i |^2 \, \rm{d}x \, \rm{d}t + \sum_{n\in\mathbb{N}} \frac{1}{2\lambda_n} \langle \partial_t w_i, \phi_{n,q_i}\rangle^2_{\mathrm{L}^2(\Omega)}(T^*_i)\\
		\leq & \left\Vert \partial_t w_i \right\Vert_{\mathrm{L}^2((0,T)\times\Omega)} + \left\Vert \partial_t w_i \right\Vert_{\mathrm{L}^2(\Omega)}(T^*_i) < +\infty, \quad \mbox{for} \ i=1,2.
	\end{align*}
	For $i=1,2$, we can perform similar computation for $w_i$ also. Hence, we conclude the following: for $i=1,2$
	\begin{equation}\label{Hilbert norm to infty}
		\left \{
		\begin{aligned}
			w_i \in \mathrm{L}^2((0,+\infty); \mathrm{L}^2(\Omega)), & \ \ q_iw_i \in \mathrm{L}^2((0,+\infty); \mathrm{L}^2(\Omega))\\
			\partial_t w_i \in \mathrm{L}^2((0,+\infty); \mathrm{L}^2(\Omega)), & \ \ (-\Delta)^a w_i \in \mathrm{L}^2((0,+\infty); \mathrm{L}^2(\Omega)).
		\end{aligned}
		\right .
	\end{equation}
	In the next proposition we will show that the two DN maps corresponding to equations \eqref{varying boundary data}, are same. 
	\begin{Prop}\label{equality of parabolic DN map.}
		For $i=1,2$, let $w_i$ be the solution to \eqref{varying boundary data} corresponding to the potential $q_i$. Let $a\in\left(\frac12,1\right)$ and let Borg-Levinson criteria \eqref{Borg Levinson fractional criteria} holds. Then the parabolic \emph{Dirichlet to Neumann (DN) maps} are equal i.e.,
		\[
		\partial_{\nu} \left( \frac{w_1}{d^{a-1}}\right)= \partial_{\nu} \left( \frac{w_2}{d^{a-1}}\right), \quad \forall \, t>0. 
		\] 
	\end{Prop}
	\begin{proof}
		Thanks to \eqref{Hilbert norm to infty}, we can define Laplace transformation of $w_i$, $\partial_t w_i$, $q_iw_i$, $(-\Delta)^a w_i$. Laplace transformation is defined in the following way:
		\[
		\mathcal{L}(\xi)(s)= \int_{0}^{+\infty}e^{-st}\xi(t) \, \rm{d}t,
		\]
		where $\xi\in \mathrm{L}^2(0,+\infty)$. The integral is well defined in the upper half complex plane and  Paley-Wiener theorem  ensures that the function is analytic in the same upper half complex plane. From the property of Laplace transformation and Fubini's theorem, we have the following results
		\begin{equation}\label{properties of laplace transformation}
			\left \{
			\begin{aligned}
				\mathcal{L}(\partial_t w_i)(s)=& -s\mathcal{L} (w_i)(s)\\ \mathcal{L}((-\Delta)^a w_i)(s)=& (-\Delta)^a (\mathcal{L} (w_i))\\
				\mathcal{L} (q_i w_i)(s)=& q_i \mathcal{L} (w_i)(s)\\
				\mathcal{L}\left(\frac{w_i}{d^{a-1}}\right)=& \mathcal{L}(f)= \frac{\mathcal{L}(w_i)}{d^{a-1}},  \quad \mbox{for} \ i=1,2.
			\end{aligned} 
			\right .
		\end{equation}
		We have that the quantity $\displaystyle{\partial_{\nu}\left(\frac{w_i}{d^{a-1}}\right)\in \mathrm{L}^2((0,T);H^{a-\frac{1}{2}}(\partial\Omega)}$ [see \cite{Grubb2020ExactGF}, \cite{grubb2023resolvents}, \cite{grubb2018green}]. As $a>\frac12$, the continuous inclusion $H^{a-\frac12}\hookrightarrow \mathrm{L}^2$ further yields 
		\[
		\partial_{\nu}\left(\frac{w_i}{d^{a-1}}\right)\in \mathrm{L}^2((0,T); \mathrm{L}^2(\partial\Omega)).
		\]
		Thanks to \eqref{DN continuity}, we have that
		\[
		\left \Vert \partial_{\nu}\left(\frac{w_i}{d^{a-1}}\right) \right\Vert_{\mathrm{L}^2((0,T); \mathrm{L}^2(\partial\Omega))} \leq C_{DN,i} \Vert f\Vert_{\mathrm{L}^2((0,T); H^{a+\frac12}(\partial\Omega))}, 
		\]
		where, for $i=1,2$, $C_{DN,i}$ are positive constants independent of the parameter `$T$'. We use the continuous inclusion $H^2\hookrightarrow H^{a+\frac12}$. As $f\in C_c^{\infty}((0,T)\times\partial\Omega)$, we have that
		\begin{align*}
			\left \Vert \partial_{\nu}\left(\frac{w_i}{d^{a-1}}\right) \right\Vert_{\mathrm{L}^2((0,T); \mathrm{L}^2(\partial\Omega))} \leq & C_{DN,i}C(\Omega) \Vert f\Vert_{\mathrm{L}^2((0,T); H^{2}(\partial\Omega))}\\
			\leq & C_{DN, i}C(\Omega) \left\Vert \sup_{x\in\Omega}|f|+\sup_{x\in\Omega}|Df|+\sup_{x\in\Omega}|D^2f| \right\Vert_{\mathrm{L}^2(0,T)},
		\end{align*}
		where $Df$ and $D^2f$ denotes the gradient and Hessian  of $f$ and $C(\Omega)$ is a positive constant depending on the domain only. Thanks to the compact support of $f$, we have
		\[
		\left \Vert \partial_{\nu}\left(\frac{w_i}{d^{a-1}}\right) \right\Vert_{\mathrm{L}^2((0,+\infty); \mathrm{L}^2(\partial\Omega))} \leq C_{DN, i}C(\Omega) \left\Vert \sup_{x\in\Omega}|f|+\sup_{x\in\Omega}|Df|+\sup_{x\in\Omega}|D^2f| \right\Vert_{\mathrm{L}^2(0,T)}.
		\] 
		Hence Laplace transformation is well defined and analytic in upper half complex plane. Applying the Laplace transform to equation \eqref{varying boundary data}, we obtain: for $i=1,2$:
		\begin{equation*}
			\left \{
			\begin{aligned}
				(-\Delta)^a \mathcal{L}(w_i)+q_i \mathcal{L}(w_i)= & s\mathcal{L}(w_i) \qquad \mbox{in}\ \Omega\\
				\mathcal{L}(w_i)=& 0 \qquad \qquad \ \mbox{in} \ \mathbb{R}^N\setminus\overline{\Omega}\\
				\frac{\mathcal{L}(w_i)}{d^{a-1}}=& \mathcal{L}(f) \qquad \ \ \ \mbox{on} \ \partial\Omega.
			\end{aligned}
			\right .
		\end{equation*}
		Let us denote the upper half complex plane by $\mathcal{H}_+$. Then, thanks to the Borg-Levinson condition \eqref{Borg Levinson fractional criteria} and Proposition \ref{equality of elliptic DN map}, for all $s\in\mathcal{H_+}$, The \emph{Dirichlet to Neumann (DN) maps} are the same, i.e.,
		\begin{align}\label{Laplace Dirichlet DN map same}
			\partial_{\nu} \left( \frac{\mathcal{L}w_1}{d^{a-1}}\right)= \partial_{\nu} \left( \frac{\mathcal{L}w_2}{d^{a-1}}\right).
		\end{align}
		We claim that $\displaystyle{\partial_{\nu} \left( \frac{\mathcal{L}w_i}{d^{a-1}}\right)= \mathcal{L} \left(\partial_{\nu} \left( \frac{w_i}{d^{a-1}}\right)\right)}$, for $i=1,2$. We use integration by parts formula. Let $\zeta_i$, satisfies the following equation: for $i=1,2$
		\begin{equation*}
			\left \{
			\begin{aligned}
				(-\Delta)^a \zeta_i +q_i \zeta_i= & 0 \qquad \mbox{in} \ \Omega\\
				\zeta_i=& 0 \qquad \mbox{in} \ \mathbb{R}^N\setminus\overline{\Omega}\\
				\frac{\zeta_i}{d^{a-1}}=& g \qquad \mbox{on} \ \partial\Omega, 
			\end{aligned}
			\right .
		\end{equation*}
		where $g\in C_c^{\infty}(\Omega)$. Hence the following regularity holds [see \cite{grubb2014local}]
		\[
		\zeta_i \in e^+ d(x)^a C^{\infty}(\overline{\Omega}) \in C(\overline{\Omega}).
		\]
		Employing integration by parts formula \eqref{integration by parts}, we obtain
		\begin{align*}
			&\left\langle \big((-\Delta)^a+q_i\big) w_i, \zeta_i\right\rangle_{\mathrm{L}^2(\Omega)}  - \left\langle \big((-\Delta)^a+q_i\big) \zeta_i, w_i\right\rangle_{\mathrm{L}^2(\Omega)}\\
			=& \Gamma(a)\Gamma(a+1)\left(\left\langle \partial_{\nu} \left( \frac{w_i}{d^{a-1}}\right), \frac{\zeta_i}{d^{a-1}}\right\rangle_{\mathrm{L}^2(\partial\Omega)} - \left\langle \partial_{\nu} \left( \frac{\zeta_i}{d^{a-1}}\right), \frac{w_i}{d^{a-1}}\right\rangle_{\mathrm{L}^2(\partial\Omega)}\right),
		\end{align*}
		for $i=1,2$, and for all $t>0$. Thanks to the fact that $u\in C(\overline{\Omega})$, $\displaystyle{\frac{\zeta_i}{d^{a-1}}\in C_c^{\infty}{\Omega}}$ and $\displaystyle{\partial_{\nu} \left( \frac{\zeta_i}{d^{a-1}}\right)\in \mathrm{L}^2(\Omega)}$, application of Fubini's theorem, yields the following results: for $i=1,2$
		\begin{equation*}
			\left \{
			\begin{aligned}
				&\mathcal{L}\left(\left\langle \big((-\Delta)^a+q_i\big) w_i, \zeta_i\right\rangle_{\mathrm{L}^2(\Omega)}\right) = \left\langle \big((-\Delta)^a+q_i\big) \mathcal{L}(w_i), \zeta_i\right\rangle_{\mathrm{L}^2(\Omega)}\\
				&\mathcal{L}\left(\left\langle \big((-\Delta)^a+q_i\big) \zeta_i, w_i\right\rangle_{\mathrm{L}^2(\Omega)}\right)= \left\langle \big((-\Delta)^a+q_i\big) \zeta_i, \mathcal{L}(w_i)\right\rangle_{\mathrm{L}^2(\Omega)}\\
				&\mathcal{L}\left(\left\langle \partial_{\nu} \left( \frac{w_i}{d^{a-1}}\right), \frac{\zeta_i}{d^{a-1}}\right\rangle_{\mathrm{L}^2(\partial\Omega)}\right)=\left\langle \mathcal{L}\left(\partial_{\nu} \left( \frac{w_i}{d^{a-1}}\right)\right), \frac{\zeta_i}{d^{a-1}}\right\rangle_{\mathrm{L}^2(\partial\Omega)}\\
				&\mathcal{L}\left(\left\langle \partial_{\nu} \left( \frac{\zeta_i}{d^{a-1}}\right), \frac{w_i}{d^{a-1}}\right\rangle_{\mathrm{L}^2(\partial\Omega)}\right)= \left\langle \partial_{\nu} \left( \frac{\zeta_i}{d^{a-1}}\right), \frac{\mathcal{L}(w_i)}{d^{a-1}}\right\rangle_{\mathrm{L}^2(\partial\Omega)}.
			\end{aligned}
			\right .
		\end{equation*}
		Hence we obtain the following result
		\begin{align*}
			&\left\langle \big((-\Delta)^a+q_i\big) \mathcal{L}(w_i), \zeta_i\right\rangle_{\mathrm{L}^2(\Omega)}-\left\langle \big((-\Delta)^a+q_i\big) \zeta_i, \mathcal{L}(w_i)\right\rangle_{\mathrm{L}^2(\Omega)}\\
			=& \Gamma(a)\Gamma(a+1) \left(\left\langle \mathcal{L}\left(\partial_{\nu} \left( \frac{w_i}{d^{a-1}}\right)\right), \frac{\zeta_i}{d^{a-1}}\right\rangle_{\mathrm{L}^2(\partial\Omega)}- \left\langle \partial_{\nu} \left( \frac{\zeta_i}{d^{a-1}}\right), \frac{\mathcal{L}(w_i)}{d^{a-1}}\right\rangle_{\mathrm{L}^2(\partial\Omega)} \right), 
		\end{align*}
		for $i=1,2$. As $\frac{\zeta_i}{d^{a-1}}\in C_c^{\infty}(\Omega)$ arbitrary, we can conclude
		\[
		\mathcal{L}\left(\partial_{\nu} \left( \frac{w_i}{d^{a-1}}\right)\right)= \partial_{\nu} \left( \frac{\mathcal{L}w_i}{d^{a-1}}\right), \quad \mbox{for} \ i=1,2.
		\]
		Finally from \eqref{Laplace Dirichlet DN map same}, we conclude
		\[
		\mathcal{L}\left(\partial_{\nu} \left( \frac{w_1}{d^{a-1}}\right)\right)=\mathcal{L}\left(\partial_{\nu} \left( \frac{w_2}{d^{a-1}}\right)\right).
		\]
		Lastly, injectivity of Laplace transformation in $\mathrm{L}^2$ helps us conclude
		\[
		\partial_{\nu} \left( \frac{w_1}{d^{a-1}}\right)= \partial_{\nu} \left( \frac{w_2}{d^{a-1}}\right), \quad \forall \, t>0.
		\]
	\end{proof}
	In the next section, we recover the potentials from the non-local parabolic equations through the equality of the \emph{Dirichlet to Neumann (DN) maps}.
	
	\section{Parabolic inverse problem, Neumann data} \label{Parabolic inverse problem Neumann data}
	
	We use density of solution space at a particular time slice \cite{das2025boundarycontrolcalderontype} to analyze inverse problem on the non-local parabolic equation through the equality of \emph{Dirichlet to Neumann (DN) maps}. For the convenience of the readers, we will put the density results in the Appendix \ref{Appendix}. Here, the density of the solution space plays a crucial role. To obtain the density, we assume certain smallness assumption on one of the potential. The precise result is the following: 
	\begin{Prop} \label{non local parabolic inverse problem}
		For $i=1,2$, let $w_i$ satisfies:
		\begin{equation}\label{varying boundary data, time potential}
			\left\{
			\begin{aligned}
				\partial_t w_{i}+(-\Delta)^a w_i+ q_iw_i=&0 \qquad \qquad  \mbox{in}\ (0,T)\times\Omega\\
				w_i=&0 \qquad \qquad  \mbox{in} \ (0,T)\times\mathbb{R}^N\setminus\overline{\Omega}\\
				\lim\limits_{\substack{x\to y\\ y\in\partial\Omega}}\frac{w_i}{d^{a-1}}(y)=& f \qquad \qquad \mbox{on} \ (0,T)\times\partial\Omega\\
				w_i(0,\cdot)=&0 \qquad \qquad  \mbox{in} \ \Omega.
			\end{aligned}
			\right .
		\end{equation}
		Here $\Omega\subset\mathbb{R}^N$ is a bounded $C^{1,1}$ domain with boundary $\partial\Omega$ and $a\in(\frac12,1)$. The potentials $q_i\in C_c^{\infty}\left((0,T)\times\Omega\right)$ for $i=1,2$, and $q_i$ is chosen in such a way that zero is not an eigenvalue for  \eqref{varying boundary data, time potential}, corresponding to the case $f\equiv0$. Let $\displaystyle{\theta< \frac{1}{2}\left(1+C_{HS}\left(\frac N2+R\right)\right)^{-1}}$, where $C_{HS}$ is the Hardy-Sobolev constant on the domain $\Omega$ and $R:=\max\{\Vert x\Vert, x\in\Omega\}$. Let the potential $q_1$ is non-negative and $|q_1(x)|,|\nabla q_1(x)|\leq \theta$ for all $x\in\Omega$. Furthermore, for \eqref{varying boundary data, time potential}, we assume the the singular boundary data $f\in C_c^{\infty}((0,T)\times\partial\Omega)$. Then, if the \emph{Dirichlet to Neumann (DN) maps} are equal
		\[
		\partial_{\nu}\left(\frac{w_1}{d^{a-1}}\right)= \partial_{\nu}\left(\frac{w_2}{d^{a-1}}\right), \qquad \mbox{for a.e.} \ (t,x)\in(0,T)\times\partial\Omega,
		\]
		we have
		\[
		q_1\equiv q_2.
		\]
	\end{Prop}  
	\begin{proof}
		Let $w_i$ be the solution to \eqref{varying boundary data, time potential} for $i=1,2$. Thanks to Theorem \ref{s-transmission regularity, boundary data},  we have that 
		\[
		w_i\in \mathrm{L}^2\left((0,T); H^{(a-1)(2a)}(\overline{\Omega})\right) \cap \overline{H}^1\left((0,T); \mathrm{L}^2(\Omega)\right), \ \mbox{for} \ i=1,2.
		\]
		Let $v=w_1-w_2$ satisfy
		\begin{equation*}\label{difference equation}
			\left \{
			\begin{aligned}
				\partial_t v+(-\Delta)^av+(q_1w_1-q_2w_2)=& 0 \qquad \mbox{in} \ (0,T)\times\Omega\\
				v= & 0 \qquad \mbox{in} \ (0,T)\times\mathbb{R}^N\setminus\overline{\Omega}\\
				\frac{v}{d^{a-1}}=& 0 \qquad \mbox{on} \ (0,T)\times\partial\Omega\\
				v(0,\cdot)=& 0 \qquad \mbox{in} \ \{0\}\times\Omega. 
			\end{aligned}
			\right .
		\end{equation*}	
		Let $V$ satisfy
		\begin{equation*}\label{Auxiliary test equation}
			\left \{
			\begin{aligned}
				\partial_t V+(-\Delta)^aV=& 0 \qquad \mbox{in} \ (0,T)\times\Omega\\
				V= & 0 \qquad \mbox{in} \ (0,T)\times\mathbb{R}^N\setminus\overline{\Omega}\\
				\frac{V}{d^{a-1}}=& F \qquad \mbox{on} \ (0,T)\times\partial\Omega\\
				V(0,\cdot)=& 0 \qquad \mbox{in} \ \{0\}\times\Omega, 
			\end{aligned}
			\right .
		\end{equation*}	
		where $F(0,\cdot)=0$ and $F\in \mathrm{L}^2((0,T); H^{a+\frac{1}{2}}(\partial\Omega))$. Hence thanks to Theorem \ref{s-transmission regularity, boundary data} \cite{grubb2023resolvents}, we have that 
		\[
		V \in \mathrm{L}^2\left((0,T); H^{(a-1)(2a)}(\overline{\Omega})\right) \cap \overline{H}^1\left((0,T); \mathrm{L}^2(\Omega)\right).
		\] 
		Integration by parts \eqref{integration by parts} yields
		\begin{align*}
			\int_{\Omega} &V(s,x)(-\Delta)^a v(t,x) \, \rm{d}x-\int_{\Omega} v(t,x)(-\Delta)^a V(t,x) \, \rm{d}x\\
			=
			&\Gamma(a)\Gamma(a+1) \int_{\partial\Omega} \frac{V}{d^{a-1}}(s,\sigma)\partial_{\nu}\left(\frac{v}{d^{a-1}}\right)(t,\sigma)- \frac{v}{d^{a-1}}(t,\sigma)\partial_{\nu}\left(\frac{V}{d^{a-1}}\right)(s,\sigma) \, \rm{d}\sigma=0.  
		\end{align*}
		Integrating with respect to variable `$t$' yields
		\[
		\int_0^T \int_{\Omega} V(s,x)(-\Delta)^a v(t,x)- v(t,x)(-\Delta)^a V(t,x) \, \rm{d}x \, \rm{d}t=0
		\]
		Equation corresponding to $v$ and $V$ yields
		\[
		\int_{0}^T \int_{\Omega} \left( -\partial_t v(t,x) -(q_1w_1-q_2w_2)(t,x)\right) V(s,x) \, \rm{d}x\, \rm{d}t= \int_{0}^T \int_{\Omega} \left( -\partial_s V(s,x)\right) v(t,x) \, \rm{d}x \, \rm{d}t. 
		\]
		Simplifying, it implies
		\begin{align}\label{Neumann inverse problem parabolic: intermedite 1}
			-\int_{\Omega} v(T,x) V(s,x) \, \rm{d}x +& \int_{0}^{T}\int_{\Omega} v(t,x) \partial_s V(s,x) \, \rm{d}x\, \rm{d}t\\
			=& \int_{0}^T \int_{\Omega} (q_1w_1-q_2w_2)(t,x) V(s,x)\, \rm{d}x\, \rm{d}t \quad \mbox{for all} \ s>0. \nonumber
		\end{align}
		Consider the function $w:(0,+\infty)\times(0,+\infty)\to\mathbb{R}$, denoted by 
		\[
		w(t,s)=\int_{\Omega} v(t,x)V(s,x).
		\]
		We will show `$w$' satisfies a particular transport equation. We compute the term $\partial_t w-\partial_s w$.
		\begin{align*}
			\partial_t w-\partial_s w = \int_{\Omega} \partial_t v(t,x) & V(s,x) \, \rm{d}x -\int_{\Omega} v(t,x) \partial_s V(s,x) \, \rm{d}x\\
			= - \int_{\Omega} V(s,x)&(-\Delta)^a v(t,x) \, \rm{d}x+ \int_{\Omega} v(t,x)(-\Delta)^a V(s,x) \, \rm{d}x\\
			-& \int_{\Omega} (q_1w_1-q_2w_2)(t,x)V(s,x)\, \rm{d}x
		\end{align*} 
		Employing integration by parts formula \eqref{integration by parts}, we obtain
		\begin{align*}
			\partial_t w-&\partial_s w=   -\int_{\Omega} (q_1w_1-q_2w_2)(t,x)V(s,x)\, \rm{d}x\\
			+ & \Gamma(a)\Gamma(a+1) \int_{\partial\Omega} \frac{V}{d^{a-1}}(s,\sigma)\partial_{\nu}\left(\frac{v}{d^{a-1}}\right)(t,\sigma)- \frac{v}{d^{a-1}}(t,\sigma)\partial_{\nu}\left(\frac{V}{d^{a-1}}\right)(s,\sigma) \, \rm{d}\sigma.
		\end{align*}
		As $\displaystyle{\frac{v}{d^{a-1}}\Big|_{\partial\Omega}=\partial_{\nu}\left(\frac{v}{d^{a-1}}\right)\Big|_{\partial\Omega}=0}$, we obtain
		\[
		\partial_t w-\partial_s w=   -\int_{\Omega} (q_1w_1-q_2w_2)(t,x)V(s,x)\, \rm{d}x.
		\]
		Hence `$w$' satisfy the following transport equation 
		\begin{equation*}
			\left \{
			\begin{aligned}
				\partial_t w- \partial_s w= & -\int_{\Omega} (q_1w_1-q_2w_2)(t,x)V(s,x)\, \rm{d}x \qquad \mbox{in}\ (0,+\infty)\times(0,+\infty)\\
				w(0,s)=& 0 \qquad \qquad \qquad \qquad\qquad \qquad \qquad \ \qquad \mbox{on}\ \{0\}\times(0,+\infty)\\
				w(t,0)=& 0 \qquad \qquad \qquad \qquad\qquad \qquad \qquad \ \qquad \mbox{on}\ (0,+\infty)\times\{0\}.
			\end{aligned}
			\right .
		\end{equation*}
		We perform a change of variable `$t-s\to\xi$' and `$s\to s$'. We can rewrite the equation corresponding to `$w$' as 
		\[
		\partial_{\xi} w= -\int_{\Omega} (q_1w_1-q_2w_2)(\xi+s,x)V(s,x)\, \rm{d}x.
		\]
		Hence for any $s\geq 0$, `$w$' satisfies the following ordinary differential equation: 
		\begin{equation*}
			\left\{
			\begin{aligned}
				\partial_{\xi} w=& -\int_{\Omega} (q_1w_1-q_2w_2)(\xi+s,x)V(s,x)\, \rm{d}x \qquad \mbox{in} \ (-s,+\infty)\\
				w(-s)=& 0.
			\end{aligned}
			\right .
		\end{equation*}
		This implies
		\[
		w(\xi',s)= -\int_{-s}^{\xi'}\int_{\Omega} (q_1w_1-q_2w_2)(\xi+s,x)V(s,x)\, \rm{d}x\, \rm{d}t.
		\]
		Let $\xi'+s=T$, then change of variable yields
		\begin{align*}
			w(T,s)=&- \int_{0}^{T} \int_{\Omega} (q_1w_1-q_2w_2)(t,x) V(s,x)\, \rm{d}x\, \rm{d}t\\
			=& \int_{\Omega} v(T,x)V(s,x)\, \rm{d}x \qquad \mbox{for all}\ (T,s)\in [0,+\infty)\times[0,+\infty).
		\end{align*}
		Thanks to \eqref{Neumann inverse problem parabolic: intermedite 1}, we obtain
		\[
		\int_{0}^{T}\int_{\Omega} v(t,x) \partial_s V(s,x)\, \rm{d}x\, \rm{d}t=0.
		\]
		Integrating the above relation with respect to `$s$' variable, we obtain
		\[
		\int_{0}^{S} \int_{0}^{T}\int_{\Omega} v(t,x) \partial_s V(s,x)\, \rm{d}x \, \rm{d}t \, \rm{d}s=0.
		\]
		Hence
		\[
		\int_0^T \int_{\Omega} v(t,x)V(S,x)\, \rm{d}x\, \rm{d}t=0 \qquad \mbox{for all}\ S\geq 0.
		\]
		We have $v\in\mathrm{L}^2\left((0,T); H^{(a-1)(2a)}(\overline{\Omega})\right)\subset \mathrm{L}^2((0,T); \mathrm{L}^2(\Omega))$. Hence the function $\displaystyle{\int_{0}^T v(t,x)\, \rm{d}t\in \mathrm{L}^2(\Omega)}$ and 
		\[
		\int_{\Omega} V(S,x)\left( \int_0^T v(t,x)\, \rm{d}t\right)\, \rm{d}x=0, \quad \mbox{for all}\ S\geq 0.
		\]
		We fix a particular $S>0$. then If we vary $F$, the set $\big\{V(S,x): F\in C_c^{\infty}((0,T)\times\partial\Omega)\big\}$ is dense in $\mathrm{L}^2(\Omega)$ [see \cite{das2025boundarycontrolcalderontype}, also see Appendix \ref{Appendix}: Lemma \ref{A2}]. Hence
		\[
		\int_0^T
		v(t,x)\, \rm{d}t=0, \qquad \mbox{for all} T\geq 0.
		\]
		We fix a particular time $t=\tilde{T}>0$. We have
		\[
		\int_0^{\tilde{T}+h}
		v(t,x)\, \rm{d}t- \int_0^{\tilde{T}-h}
		v(t,x)\, \rm{d}t=0,
		\]
		where $h$ is sufficiently small. It yields 
		\[
		\frac{1}{2h} \int_{\tilde{T}-h}^{\tilde{T}+h} v(t,x)\, \rm{d}t=0.
		\]
		Employing Lebesgue differentiation theorem, we obtain
		\[
		v(\tilde{T},x)\equiv 0, \qquad \mbox{for all}\ (\tilde{T},x)\in[0,+\infty)\times\Omega.
		\]  
		hence
		\[
		w_1\equiv w_2.
		\]
		From the equation \eqref{varying boundary data, time potential}, we have that 
		\[
		(q_1-q_2)w_1\equiv 0.
		\]
		We fix $t=\tilde{T}> 0$, we have that
		\[
		\int_{\Omega} (q_1-q_2)(\tilde{T},x)w_1(\tilde{T},x)\, \rm{d}x=0.
		\]
		Thanks to the assumption on the potential `$q_1$',  if we vary $f\in C_c^{\infty}((0,T)\times\partial\Omega)$, the set $\displaystyle{\big\{w_1(\tilde{T},x): f\in C_c^{\infty}((0,T)\times\partial\Omega)\big\}}$ is dense in $\mathrm{L}^2(\Omega)$ [see \cite{das2025boundarycontrolcalderontype}, also see Appendix \ref{Appendix}: Lemma \ref{A2}]. Hence
		\[
		(q_1-q_2)(\tilde{T},x)=0, \qquad \mbox{for a.e.} \ x\in\Omega.
		\]
		Since $\tilde{T}>0$ is arbitrary, we obtain
		\[
		q_1\equiv q_2.
		\]
	\end{proof} 
	
	\begin{proof}[Proof of Theorem \ref{fractional borg-levinson theorem}:]
		Proposition \ref{equality of elliptic DN map}, Proposition \ref{equality of parabolic DN map.} and Proposition \ref{non local parabolic inverse problem} directly yields Theorem \ref{fractional borg-levinson theorem}.
	\end{proof}
	
	\vspace{.3cm}
	\textbf{Acknowledgment:} 
	The First author was funded by the Department of Atomic Energy $($DAE$)$, Government of India, in the form of Postdoctoral Research Fellowship.
	

	\vspace{.3cm}
	\textbf{Data Availability:} Data sharing is not applicable to this article as no datasets were generated or analyzed during the current study. 
	
	\vspace{.3cm}
	\textbf{Declaration:}

	\vspace{.3cm}
	\textbf{Conflict of interest:} The author declares no conflicts of interest related to this article.
	
	\bibliography{inverse.bib}
	\appendix
	\section{} \label{Appendix}	
	For the convenience of the readers, we prove the density result here. These calculations can be found in \cite{das2025boundarycontrolcalderontype}. In all the proofs we consider our domain $\Omega\subset\mathbb{R}^N$ to be smooth, bounded and connected domain. We start with a particular adjoint problem and states some results corresponding to well-posedness and regularity. Let $u_{f^*}$ satisfy: 
	\begin{equation}\label{adjoint system initial data}
		\left \{
		\begin{aligned}
			-\partial_t u_{f^*}+(-\Delta)^a u_{f^*}+ qu_{f^*}=&0 \qquad \qquad  \mbox{in}\ (0,T)\times\Omega\\
			u_{f^*}=&0 \qquad \qquad  \mbox{in} \ (0,T)\times\mathbb{R}^N\setminus\Omega \\
			u_{f^*}(T,\cdot)=&f^* \qquad \qquad  \mbox{in} \ \Omega,
		\end{aligned}
		\right .
	\end{equation}
	where $f^*\in \mathrm{L}^2(\Omega)$. The following wellposedness result is established in \cite{grubb2023resolvents}.
	\begin{Thm}[\cite{grubb2023resolvents}]\label{s-transmission regularity, initial data}
		Let $f^*\in \mathrm{L}^2(\Omega)$, then the solution $u_{f^*}$ to \eqref{adjoint system initial data} has the following regularity:
		\[
		u_{f^*}\in \mathrm{L}^2\left((0,T); H^{a(2a)}(\overline{\Omega})\right).
		\]
	\end{Thm}
	In \cite{fernandez2016boundary}, authors have proved the following results corresponding to \eqref{adjoint system initial data}:
	\begin{itemize}
		\item [$($a$)$] for each $t_0<T$
		\begin{align}\label{H0}
			\sup\limits_{t\geq t_0}\Vert u_{f^*}(t,\cdot)\Vert_{C_c^a(\mathbb{R}^N)}\leq C_1(t_0)\Vert f\Vert_{\mathrm{L}^2(\Omega)},
		\end{align}
		\item [$($b$)$] for each $t_0<T$
		\begin{align}\label{H1}
			\sup\limits_{t\geq t_0}\left\Vert \frac{u_{f^*}(t,\cdot)}{d^a}\right\Vert_{C^{a-\epsilon}\left(\overline{\Omega}\right)}\leq C_1(t_0)\Vert f\Vert_{\mathrm{L}^2(\Omega)},
		\end{align}
		\item [$($c$)$] for each $t_0<T$
		\begin{align}\label{H4}
			\sup\limits_{t\geq t_0} \left\Vert \frac{d^j}{dt^j}u_{f^*}(t,\cdot)\right\Vert_{C^a_c(\mathbb{R}^N)} \leq C_{1,j}(t_0) \Vert f\Vert_{\mathrm{L}^2(\Omega)}
		\end{align}
	\end{itemize}
	for any $\epsilon\in(0,a)$. Here $C_1(t_0)$,$C_{1,j}(t_0)$ are positive constants. Authors used various eigenfunction estimates [see \cite{grubb2015spectral}, \cite{grubb2015fractional}, \cite{abatangelo2023singular}]. The following interior regularity results can be found in  \cite{fernandez2016boundary}.
	\begin{Thm}[\cite{fernandez2016boundary}]
		Let $u_{f^*}\in\mathrm{L}^2\left((0,T); H^{a}(\Omega)\right)$ be the solution to \eqref{adjoint system initial data} with $f^*\in \mathrm{L}^2(\Omega)$. Then the following holds.
		\begin{itemize}
			\item [$\bullet$:] For $t<T$, $u_{f^*}(t,\cdot)\in C^a(\mathbb{R}^N)$ and, for every $\beta\in[a,1+2a)$, $u_{f^*}(t,\cdot)$ is of class $C^{\beta}(\Omega)$ and 
			\begin{align}\label{H2}
				[u_{f^*}(t,\cdot)]_{C^{\beta}(\Omega_{\delta})} \leq C \delta^{a-\beta}, \forall \delta\in(0,1).
			\end{align}
			\item [$\bullet$:] The function $\displaystyle{\frac{u_{f^*}(t,\cdot)}{d^a}}\Big|_{\Omega}$ can be continuously extended to $\overline{\Omega}$. Moreover, there exists $\alpha\in(0,1)$ such that $\displaystyle{\frac{u_{f^*}}{d^a}\in C^{\alpha}(\overline{\Omega})}$. In addition, for $\beta\in[\alpha,a+\alpha]$, the following estimate holds
			\begin{align}\label{H3}
				\left[\frac{u_{f^*}}{d^a}\right]_{C^{\beta}(\Omega_{\delta})} \leq C \delta^{a-\beta}, \forall \delta\in(0,1).
			\end{align}
		\end{itemize}
		Here $\Omega_{\delta}:=\{x\in\Omega: d(x,\partial\Omega)\geq\delta\}$ and $[\cdot]_{C^{\beta}(\Omega_{\delta})}$ denotes the H\"older seminorm:
		\[
		\left[\frac{u_{f^*}(t,x)}{d^a}\right]_{C^{\beta}(\Omega_{\delta})} =\sup\limits_{x,y\in C^{\beta}(\Omega_{\delta}), x\neq y} \frac{\vert D^{\floor{\beta}}u_{f^*}(t,x)-D^{\floor{\beta}}u_{f^*}(t,y)\vert}{\vert x-y\vert^{\beta-\floor{\beta}}}
		\]
		where $\floor{\beta}$, denotes the greatest integer function i.e $\floor{\beta}:=\sup\{z\in\mathbb{Z}: z\leq \beta\}$. 
	\end{Thm}
	With \eqref{H0}, \eqref{H1}, \eqref{H2}, \eqref{H3} in hand, the solution $u_{f^*}$ corresponds to \eqref{adjoint system initial data} satisfies the fractional version of Pohozaev identity. For all $t>0$:
	\begin{align}\label{Phozaev identity}
		\int_{\Omega}(x\cdot \nabla u_{f*})(-\Delta)^a u_{f^*} \, \rm{d}x= \frac{2a-N}{2}\int_{\Omega}&u_{f^*}(-\Delta)^a u_{f^*} \, \rm{d}x\\
		&- \frac{\Gamma(1+a)^2}{2}\int_{\partial\Omega}\left(\frac{u_{f^*}}{d^a}\right)^2(x\cdot\nu) \, \rm{d}\sigma,\nonumber
	\end{align}
	where $\nu$ is the unit outward normal to $\partial\Omega$ at the point $x$ and $\Gamma$ is the gamma function. Proof of this statement is established in the articles \cite{ros2014pohozaev}, \cite{fernandez2016boundary}. Next we prove an observability estimate corresponding to \eqref{adjoint system initial data}. The proof of the following proposition can be found in \cite{das2025boundarycontrolcalderontype}, \cite{biccari2025boundary}.
	\begin{Prop}\label{observability estimate}
		Let $a\in\left(\frac12, 1\right)$ and let $f^*\in \mathrm{L}^2(\Omega)$ and let $0\leq \epsilon< T$. Let $u_{f^*}$ be the solution to \eqref{adjoint system initial data} with initial data $f^*$. Let $\displaystyle{\theta< \frac{1}{2}\left(1+C_{HS}\left(\frac N2+R\right)\right)^{-1}}$, where $C_{HS}$ is the Hardy-Sobolev constant on the domain $\Omega$ and $R:=\max\{\Vert x\Vert, x\in\Omega\}$. Let the potential $q$ is non-negative and $|q(x)|,|\nabla q(x)|\leq \theta$ for all $x\in\Omega$. Then the following observability estimate holds 
		\[
		\mathcal{A}\int_{0}^{T-\epsilon} \int_{\partial\Omega}\left(\frac{u_{f^*}}{d^a}\right)^2(x\cdot\nu) \, \rm{d}\sigma \, \rm{d}s \geq \Vert u_{f^*}(0)\Vert_{\mathrm{L}^2(\Omega)},  
		\]
		where, the positive constant $\mathcal{A}$, depends only on the domain, the dimension, the time factor $T-\epsilon$ and the exponent `$a$'.
	\end{Prop}
	\begin{proof}
		The proof is exactly similar to the proof of observability estimate in \cite{biccari2025boundary}. The only difference is we have a small potential here, which was not present in \cite{biccari2025boundary}. The idea is to obtain a control function that renders the solution null controllable within a specific finite-dimensional subspace,
		\begin{align}\label{finite projective space}
			\mathcal{H}_{J}:=span\{\phi_1,\cdots,\phi_J\},
		\end{align}
		and then use this control as a test function in the adjoint system with the given initial data \eqref{adjoint system initial data}. Here $\{\phi_i\}\Big|_{i\in\mathbb{N}}$ is the eigenfunctions of the operator $(-\Delta)^a+q$, which generates an orthonormal basis of $\mathrm{L}^2(\Omega)$. In \cite{biccari2025boundary}, the authors achieved this by making a detour through the wave equation, establishing an observability estimate for its energy, specifically for solutions lying in the finite-dimensional subspace $\mathcal{H}_J$. Although in \cite{biccari2025boundary} authors don't have a potential `$q$'. We claim that even if we have a small potential we can obtain similar result. So, in order to prove Proposition \ref{observability estimate}, we make a little detour to the wave equation with a small potential and will obtain a observability estimate corresponding to it's energy, when the solutions are coming  from the space $\mathcal{H}_J$. 
	\end{proof}
	This proof is motivated from the article \cite{biccari2025boundary}. We follow exactly the same method as in \cite{biccari2025boundary}. We only have an extra non-negative small potential with small growth. The smallness in our calculations is explicit in nature. Consider the wave equation: 
	\begin{equation}\label{adjointWave}
		\begin{cases}
			p_{tt} + (-\Delta)^ap+qp = 0 & \mbox{in }\; (0,T)\times\Omega
			\\
			p=0 & \mbox{in }\; (0,T)\times(\mathbb R^N\setminus\Omega)
			\\
			p(\cdot,0)=p_0,\;\;p_t(\cdot,0)=p_1 & \mbox{in }\;\Omega.
		\end{cases}
	\end{equation}
	Note that, at this level, the system being time-reversible, taking the initial data at $t = 0$ or $t = T$ is irrelevant, contrarily to the parabolic setting.
	
	The system \eqref{adjointWave}, governed by the non-local fractional operator $(-\Delta)^a$, which is symmetric maximal monotone operator and generates group of isometries. As a consequence, it admits a unique weak solution that preserves the total energy (see, for instance, [Chapter 10, Theorem 10.14] in \cite{brezis2011functional}). In other words, for every initial data $(p_0,p_1)\in H_0^a(\Omega)\times L^2(\Omega)$ there exists a unique finite energy solution 
	\begin{align*}
		p\in C([0,T];H_0^a(\Omega))\cap C^1([0,T];\mathrm{L}^2(\Omega))\cap C^2([0,T];H^{-a}(\Omega))
	\end{align*}
	of \eqref{adjointWave}. Here $H_0^a(\Omega)$ denotes the fractional order Sobolev space consisting of all functions in $H^a(\mathbb{R}^N)$ vanishing in $\mathbb{R}^N\setminus\Omega$, while $H^{-a}(\Omega)$ is its dual. Moreover, the function $p$ admits the following spectral representation:
	\begin{align*}
		p(x,t) = \sum_{j\in\mathbb{N}} p_j(t)\phi_j(x)
	\end{align*}
	where, for each \( j \in \mathbb{N} \), the coefficient function $p_j$ satisfies the second-order ordinary differential equation:
	\begin{align*}
		\begin{cases}
			p_j^{\prime\prime}(t)+ \lambda_j p_j(t) = 0,\qquad t\in (0,T),
			\\
			p_j(0) = a_j, \;\;\; p_j'(0) = b_j.
		\end{cases}
	\end{align*}
	Here, $(\lambda_j,\phi_j)_{j\in\mathbb{N}}$ denote the eigenvalues and eigenfunctions of the operator $(-\Delta)^a+q$, and $(a_j,b_j)_{j\in\mathbb{N}}$ are the Fourier coefficients of the initial data $(p_0,p_1)$ with respect to the basis $(\phi_j)_{j\in\mathbb{N}}$, that is,
	\begin{align*}
		p_0(x) = \sum_{j\in\mathbb{N}} a_j\phi_j(x) \quad \mbox{and} \quad p_1(x) = \sum_{j\in\mathbb{N}} b_j\phi_j(x).
	\end{align*}
	Thanks to the symmetricity of the operator $(-\Delta)^a+q$, the energy of solutions of \eqref{adjointWave} is conserved along time.  That is, 
	\begin{align}\label{e24}
		E_a(t):=\frac 12\int_{\Omega}|p_t|^2\,dx+\frac 12\int_{\Omega}|(-\Delta)^{\frac a2}p|^2\,dx+ \frac 12\int_{\Omega} qp^2\, dx= E_a(0).
	\end{align}
	Now we establish the  multiplier identity for solutions of \eqref{adjointWave}. To this end, we focus on a special class of solutions $p\in \mathcal H_J$, with $\mathcal H_J$  as in \eqref{finite projective space}. They are finite energy solutions of \eqref{adjointWave} involving a finite number of Fourier components, of the form 
	\begin{align}\label{solFourier}
		p(x,t) = \sum_{j=1}^J p_j(t)\phi_j(x).
	\end{align}
	The following result is the Pohozaev identity for the finite energy solution of the wave equation \eqref{adjointWave} in the projective space \eqref{finite projective space}, which helps us obtain Proposition \ref{observability estimate}. 
	\begin{Prop}\label{lemmaMult}
		Let $a\in(0,1)$ and $p\in\mathcal H_J$ be the solution of \eqref{adjointWave} constituted by a finite number of Fourier components, as defined in \eqref{solFourier}. Then,
		\begin{align}\label{e210}
			\frac{\Gamma(1+a)^2}{2}\int_{\epsilon}^T \int_{\partial\Omega}(x\cdot\nu)\left|\frac{p}{d^a}\right|^2\,d\sigma dt
			&=  aTE_a(\epsilon)+ \int_{\Omega}p_t\left(x\cdot\nabla p + \frac{N-a}{2}p\right)\,dx\,\bigg|_{t=\epsilon}^{t=T}\\
			&-\int_{\epsilon}^T\!\int_{\Omega}\left(q\frac a2+x\cdot\nabla \frac{q}{2}\right) p^2 \, dxdt- aT\int_{\Omega} qp^2 \, dx\Big|_{t=\epsilon} \notag
		\end{align}
		where $\Gamma$ is the Euler-Gamma function.
	\end{Prop}
	\begin{proof}
		\noindent\textbf{Step 1: Application of Pohozaev identity.} We formally multiplying the first equation in \eqref{adjointWave} by $x\cdot\nabla p$. Integrating over $(\epsilon,T)\times\Omega$ we get that
		\begin{align}\label{A1}
			\int_{\epsilon}^T\int_{\Omega}p_{tt}\left(x\cdot\nabla p\right)\,dxdt +\int_{\epsilon}^T\int_{\Omega}(-\Delta)^ap\left(x\cdot\nabla p\right)\,dxdt+\int_{\epsilon}^T\int_{\Omega} q(x\cdot \nabla p)p =0.
		\end{align}
		The regularity assumptions of \cite{ros2014pohozaev}, are fulfilled by solutions of the fractional wave equation involving a finite number of Fourier components \cite{biccari2025boundary}, we have that, for all $t\in [\epsilon, T]$, 
		\begin{align}\label{A2}
			\int_{\Omega}(-\Delta)^ap \left(x\cdot\nabla p\right)\,dx = \frac{2a-N}{2}\int_{\Omega}|(-\Delta)^{\frac a2}p|^2\,dx - \frac{\Gamma(1+a)^2}{2}\int_{\partial\Omega}(x\cdot\nu)\left|\frac{p}{d^a}\right|^2\,d\sigma.
		\end{align}
		Using \eqref{A2} we get from \eqref{A1} that 
		\begin{align}\label{A3}
			\int_{\epsilon}^T\int_{\Omega}p_{tt}\left(x\cdot\nabla p\right)\,dxdt + \int_{\epsilon}^T\int_{\Omega} qp(x\cdot \nabla p) \, dxdt =& \frac{N-2a}{2}\int_{\epsilon}^T\int_{\Omega}|(-\Delta)^{\frac a2}p|^2\, dxdt\\
			+ &\frac{\Gamma(1+a)^2}{2}\int_{\epsilon}^T\int_{\partial\Omega}(x\cdot\nu)\left|\frac{p}{d^a}\right|^2\,d\sigma dt. \nonumber
		\end{align}
		The right hand side can be computed as follows:
		\begin{align}\label{A4}
			\int_{\epsilon}^T\!\int_{\Omega}p_{tt}&\left(x\cdot\nabla p\right)\,dxdt + \int_{\epsilon}^T\int_{\Omega} q(x\cdot \nabla p)p \, dx\\ = &\, \int_{\Omega}p_t\left(x\cdot\nabla p\right)\,dx\,\bigg|_{t=\epsilon}^{t=T} 
			- \int_{\epsilon}^T\!\int_{\Omega} p_t\left(x\cdot\nabla p_t\right)\,dxdt+\int_{\epsilon}^T\int_{\Omega} q(x\cdot \nabla p)p \, dx  \notag
			\\
			= & \int_{\Omega}p_t\left(x\cdot\nabla p\right)\,dx\,\bigg|_{t=\epsilon}^{t=T} 
			-\int_{\epsilon}^T\!\int_{\Omega}x\cdot\nabla\left(\frac{|p_{t}|^2}{2}\right)\,dxdt+\int_{\epsilon}^T\int_{\Omega} qx\cdot \nabla \left(\frac{p^2}{2}\right) \, dx  \notag
			\\
			= & \int_{\Omega} p_t\left(x\cdot\nabla p\right)\,dx\,\bigg|_{t=\epsilon}^{t=T} + \frac N2\int_{\epsilon}^T\!\int_{\Omega}|p_t|^2\,dxdt- \int_{\epsilon}^T\!\int_{\Omega}\left(q\frac N2+x\cdot\nabla \frac{q}{2}\right) p^2 \, dxdt \label{e22}  .
		\end{align}
		Once identities \eqref{A3} and \eqref{A4} are obtained and justified, combining them, we get \eqref{e22}. We refer to \cite{biccari2025boundary} for justification. The same method follows when we have a smooth compactly supported potential. 
		
		\noindent\textbf{Step 2: Equipartition of the energy}. We now derive the equipartition of the energy identity. Multiplying the first equation in \eqref{adjointWave} by $p$ and integrating by parts over $(\epsilon, T)\times\Omega$ (by using the integration by parts formula \eqref{integration by parts Hs}) we obtain 
		\begin{align}\label{e28}
			-\int_{\epsilon}^T\int_{\Omega}|p_t|^2\,dxdt + \int_{\epsilon}^T\int_{\Omega}|(-\Delta)^{\frac a2}p|^2\,dxdt + \int_{\epsilon} p_tp\,dx\,\bigg|_{t=\epsilon}^{t=T}+\int_{\epsilon}^T\int_{\Omega}q|p|^2\,dxdt=0.
		\end{align}
		Moreover, the conservation of energy yields
		\begin{align*}
			\frac N2 \int_{\epsilon}^T \int_{\Omega}|p_t|^2\,dxdt + & \frac{2a-N}{2}\int_{\epsilon}^T\int_{\Omega} |(-\Delta)^{\frac a2}p|^2\,dxdt
			\\
			= \frac a2\int_{\epsilon}^T\int_{\Omega}|p_t|^2\,dxdt &+ \frac a2\int_{\epsilon}^T\int_{\Omega} |(-\Delta)^{\frac a2}p|^2\,dxdt \\
			&+ \frac{N-a}{2}\int_{\epsilon}^T\int_{\Omega}\Big(|p_t|^2 - |(-\Delta)^{\frac a2}p|^2\Big)\,dxdt
			\\
			= aTE_a(\epsilon) +&  \frac{N-a}{2}\int_{\epsilon}^T\int_{\Omega}\Big(|p_t|^2 - |(-\Delta)^{\frac a2}p|^2\Big)\,dxdt -aT\int_{\Omega} qp^2 \, dx\Big|_{t=\epsilon}.
		\end{align*}
		Using this in the identity \eqref{A3}, we then have
		\begin{align}\label{e29}
			\frac{\Gamma(1+a)^2}{2}&\int_0^T\int_{\partial\Omega}(x\cdot\nu)\left|\frac{p}{d^a}\right|^2\,d\sigma dt\\
			=&\, aTE_a(\epsilon)+ \frac{N-a}{2}\int_0^T\int_{\Omega}\Big(|p_t|^2-|(-\Delta)^{\frac a2}p|^2\Big)\,dxdt 
			\\
			& + \int_{\Omega} p_t\left(x\cdot\nabla p\right)\,dx\,\bigg|_{t=0}^{t=T}- \int_{\epsilon}^T\!\int_{\Omega}\left(q\frac N2+x\cdot\nabla \frac{q}{2}\right) p^2 \, dxdt- aT\int_{\Omega} qp^2 \, dx\Big|_{t=\epsilon}. \notag
		\end{align}
		Combining \eqref{e28} and \eqref{e29} we get the identity \eqref{e210} and the proof is finished.
	\end{proof}
	Next we move to the proof of the observability estimate for the finite dimensional solution of the wave equation \eqref{adjointWave}. 
	\begin{Lem}\label{prop31}
		Fix $a\in(\frac 12,1)$ and $J\in\mathbb{N}$. Let $\displaystyle{\theta< \frac{1}{2}\left(1+C_{HS}\left(\frac N2+R\right)\right)^{-1}}$, where $C_{HS}$ is the Hardy-Sobolev constant on the domain $\Omega$ and $R:=\max\{\Vert x\Vert, x\in\Omega\}$. Let the potential $q$ is non-negative and  $|q(x)|,|\nabla q(x)|\leq \theta$ for all $x\in\Omega$.  Set
		\begin{equation}\label{e39}
			T_0(J) := C\lambda_J^{1-a}, 
		\end{equation}
		with a constant $C=C(N,a,\Omega)>0$.
		Then, for all $T>T_0(J)$, the following boundary observability inequality holds:
		\begin{align}\label{e37}
			E_a(T^*)\le \frac{\Gamma(1+a)^2}{2a(T-T_0(J))}\int_{\epsilon}^T\int_{\partial\Omega}(x\cdot\nu)\left|\frac{p}{d^a}\right|^2\,d\sigma dt,
		\end{align}
		for all solution $p\in \mathcal H_J$ of \eqref{adjointWave} defined in \eqref{solFourier} and for all $T^*\in[0,T]$. 
	\end{Lem}
	\begin{proof}
		$a\in(1/2,1)$. From \eqref{e210} we get that
		\begin{align*}
			\frac{\Gamma(1+a)^2}{2}\int_{\epsilon}^T \int_{\partial\Omega}(x\cdot\nu)\left|\frac{p}{d^a}\right|^2\,d\sigma dt
			&=  aTE_a(\epsilon)+ \int_{\Omega}p_t\left(x\cdot\nabla p + \frac{N-a}{2}p\right)\,dx\,\bigg|_{t=\epsilon}^{t=T}\\
			&-\int_{\epsilon}^T\!\int_{\Omega}\left(q\frac a2+x\cdot\nabla \frac{q}{2}\right) p^2 \, dxdt- aT\int_{\Omega} qp^2 \, dx\Big|_{t=\epsilon} \notag
		\end{align*}
		All the terms above is well defined \cite{biccari2025boundary}.  Let's define:
		\begin{align*}
			\xi(t):= \int_{\Omega}p_t\left(x\cdot\nabla p + \frac{N-a}{2}p\right)\,dx,
		\end{align*} 
		This is a well defined function. We borrow the following result from \cite{biccari2025boundary}. 
		\[
		\left|\xi(t)\right|\leq \texttt{C} \lambda_J^{1-a} E_a(\epsilon),
		\]
		where $\texttt{C}$ is a positive constant depending on $N,s$ and the domain $\Omega$. Thanks to Hardy Sobolev estimate, the following estimate holds
		\[
		\int_{\Omega} p^2(T) \, dx \leq  C_{HS} \int_{\Omega}(-\Delta)^{\frac a2} p^2(T) \, dx \leq E_a(\epsilon).
		\]
		Hence,
		\begin{align*}
			\left|\int_{\epsilon}^T\!\int_{\Omega}\left(q\frac a2+x\cdot\nabla \frac{q}{2}\right) p^2 \, dxdt\right| \leq & T\theta C_{HS} \left(\frac N2+R\right) E_a(\epsilon)\\
			\left|aT\int_{\Omega} qp^2 \, dx\Big|_{t=\epsilon}\right| \leq aT \theta E_s(\epsilon).
		\end{align*}
		All the above estimate yields:
		\begin{align*}
			a\Big( T-T\theta-T\theta C_{HS}\left(\frac {N}{2}+R\right) &-2\texttt{C}\lambda_J^{1-a}\Big) E_a(\epsilon) \\
			&\leq \frac{\Gamma(1+a)^2}{2}\int_{\epsilon}^T\int_{\partial\Omega}(x\cdot\nu)\left|\frac{p}{d^a}\right|^2\,d\sigma dt.
		\end{align*}
		From assumption we have $\displaystyle{\theta< \frac{1}{2}\left(1+C_{HS}\left(\frac N2+R\right)\right)^{-1}}$. Choosing the positive constant $C:=4\texttt{C}$, we conclude our proof.
	\end{proof}
	Thanks to Lemma \ref{prop31} and Proposition \ref{lemmaMult}, we can recover the results outlined in \cite{biccari2025boundary} and thereby establish Proposition \ref{observability estimate}. We now move onto the proof of our  density result.
	\begin{Lem}\label{A2}
		Let $a\in(\frac{1}{2},1)$ and let $F\in C_c^{\infty}((0,T)\times\partial\Omega)$. Let 
		\[
		\displaystyle{\theta< \frac{1}{2}\left(1+C_{HS}\left(\frac N2+R\right)\right)^{-1}},
		\]
		where $C_{HS}$ is the Hardy-Sobolev constant on the domain $\Omega$ and $R:=\max\{\Vert x\Vert, x\in\Omega\}$. Let the potential $q$ is non-negative and $|q(x)|,|\nabla q(x)|\leq \theta$ for all $x\in\Omega$. Let $u_F\in\mathrm{L}^2\left((0,T)\times\Omega\right)$ be the solution of \eqref{varying boundary data}. Then, for a fix time $\Tilde{T}>0$, the set 
		\[
		\left\{u_F(\Tilde{T},\cdot): F\in C_c^{\infty}((0,T)\times\partial\Omega)\right\} 
		\]
		is dense in $\mathrm{L}^2(\Omega)$.
	\end{Lem}
	\begin{proof}
		Let the set $\displaystyle{\{u_F(\Tilde{T},\cdot), \, F\in C_c^{\infty}((0,T)\times\partial\Omega)\}}$ is not dense in $\mathrm{L}^2(\Omega)$. Then by Hahn-Banach theorem there exists $g\in\mathrm{L}^2(\Omega)$ such that $\Vert g\Vert_{\mathrm{L}^2(\Omega)}=1$ and 
		\begin{align}\label{inner product zero, boundary data}
			\langle u_F(\Tilde{T},\cdot), g \rangle=0, \qquad \forall F\in C_c^{\infty}((0,T)\times\partial\Omega).  
		\end{align}
		Let $u_{g}\in\mathrm{L^2((0,\Tilde{T}); H^a(\Omega))}$ be the solution to the adjoint system \eqref{adjoint system initial data} with $g$ as the initial data at the time $T=\Tilde{T}$. Theorem \ref{s-transmission regularity, boundary data} ensures  $\displaystyle{u_{F}\in \mathrm{L}^2\left((0,\Tilde{T}); H^{(a-1)(2a)}(\overline{\Omega})\right)}$. We apply integration by parts formula \eqref{integration by parts dirichlt}.
		\[
		\int_{\Omega} u_{g}(-\Delta)^a u_{F}-\int_{\Omega}u_{F}(-\Delta)^au_{g} =-\Gamma(a)\Gamma(a-1)\int_{\partial\Omega} \frac{u_{F}}{d^{a-1}}\frac{u_{g}}{d^a}.
		\]
		From equation \eqref{varying boundary data}, we obtain
		\[
		\int_{\Omega} u_F\partial_t u_{g} +u_{g}\partial_t u_F = \int_{\Omega}\partial_{t}\left(u_{g}u_F\right)=-\Gamma(a)\Gamma(a-1)\int_{\partial\Omega} \frac{u_{F}}{d^{a-1}}\frac{u_{g}}{d^a} .
		\] 
		Integrating with respect to $t$ yields: 
		\[
		\int_{\Omega} u_{g}(\Tilde{T},\cdot)u_F(\Tilde{T},\cdot) \, \rm{d}x=-\Gamma(a)\Gamma(a-1)\int_{0}^{\Tilde{T}}\int_{\partial\Omega} \frac{u_{F}}{d^{a-1}}\frac{u_{g}}{d^a}\, \rm{d}x \, \rm{d}s.
		\]
		Thanks to \eqref{inner product zero, boundary data} and since $F\in C_c^{\infty}((0,T)\times\partial\Omega)$ is arbitrary, we deduce that
		\[
		\frac{u_{g}}{d^a}\equiv 0.
		\] 
		Thanks to the smallness assumption on the potential `$q$', observability estimate [Proposition \ref{observability estimate}] yields
		\[
		\Vert u_{g}\Vert_{\mathrm{L}^2((0,\Tilde{T})\times\Omega)}=0 \implies u_{g}\equiv 0 \implies g\equiv 0.
		\]
	\end{proof}
\end{document}